\documentclass[10pt]{article}
\pdfoutput = 1

\usepackage[mathscr]{eucal}
\usepackage{geometry}
\usepackage[all]{xy} 
\usepackage{graphicx}
\usepackage{caption,subcaption}
\usepackage[centertags]{amsmath}
\usepackage{amsfonts,amssymb,amsthm}
\usepackage[linesnumbered,commentsnumbered]{algorithm2e}
\usepackage[hidelinks]{hyperref}
\usepackage{mathrsfs, mathtools}
\usepackage{float}
\usepackage{color}
\usepackage{bm}
\usepackage{etoolbox}

\makeatletter
\patchcmd{\algocf@Here}{\hsize}{\linewidth}{}{}
\makeatother

\theoremstyle{plain}
\newtheorem{thm}{Theorem}[section]

\newtheorem{lemma}[thm]{Lemma}

\theoremstyle{definition}
\newtheorem{definition}[thm]{Definition}
\newtheorem{example}[thm]{Example}
\newtheorem{assumption}[thm]{Assumption}

\theoremstyle{remark}

\newcommand{\htilde}{\mathscr H}  
\newcommand{\htildefin}{\widetilde H} 
\newcommand{\hmix}{\widetilde H_\mathrm{mix}}

\newcommand{\honeoinf}{\mathscr H^1_0} 
\newcommand{\honeofin}{\widetilde H^1_0}   

\newcommand{\Qad}{Q_{\mathrm{ad}}} 
\newcommand{\An}{\mathscr A_n} 
 
\newcommand{\R}{\mathbb{R}} 
\newcommand{\Natural}{\mathbb{N}} 
\newcommand{\Span}{\mathrm{span}}
\newcommand{\prob}{\mathbb P}  
\newcommand{\proj}{\mathbf P}  
\newcommand{\Exp}{\mathbb E} 
\newcommand{\qhatn}{\hat q_n^*} 
\newcommand{\qmin}{q_{\mathrm{min}}} 
\newcommand{\qmax}{q_{\mathrm{max}}} 
\newcommand{\uhat}{\hat u} 

\begin{document}
\title{Gradient-Based Estimation of Uncertain Parameters for Elliptic Partial Differential Equations}
\author{Jeff Borggaard, Hans-Werner van Wyk}
\maketitle

\abstract{This paper addresses the estimation of uncertain distributed diffusion coefficients in elliptic systems based on noisy measurements of the model output. We formulate the parameter identification problem as an infinite dimensional constrained optimization problem for which we establish existence of minimizers as well as first order necessary conditions. A spectral approximation of the uncertain observations allows us to estimate the infinite dimensional problem by a smooth, albeit high dimensional, deterministic optimization problem, the so-called finite noise problem in the space of functions with bounded mixed derivatives. We prove convergence of finite noise minimizers to the appropriate infinite dimensional ones, and devise a stochastic augmented Lagrangian method for locating these numerically. Lastly, we illustrate our method with three numerical examples.}

\section{Introduction}
This paper discusses a variational approach to estimating the parameter $q$ in the elliptic system
\begin{equation}
-\nabla\cdot(q \nabla u) = f\  \hbox{on\ } D, \hspace{.1in} u = 0\ \hbox{on } \partial D, \label{eqn: model}
\end{equation}
based on noisy measurements $\hat u$ of $u$, when $q$ is modeled as a spatially varying random field. Equation \eqref{eqn: model}, defined over the physical domain $D\subset \mathbb R^d$, may describe the flow of fluid through a medium with permeability coefficient $q$ or heat conduction across a material with conductivity $q$. Variational formulations in which the identification problem is posed as a constrained optimization, have been studied extensively for the case when $q$ is deterministic \cite{Banks1989, Chan2003, Chan2004, Ito1990a, Ito1991}. Aleatoric uncertainty arising in these problems from imprecise, noisy measurements, variability in operating conditions, or unresolved scales are traditionally modeled as perturbations and addressed by means of regularization techniques. These approximate the original inverse problem by one in which the parameter depends continuously on the data $\hat u$, thus ensuring an estimation error commensurate with the noise level. However, when a statistical model for uncertainty in the dynamical system is available, it is desirable to incorporate this information more directly into the estimation framework to obtain an approximation not only of $q$ itself but also of its probability distribution.

\vspace{1em}

Bayesian methods provide a sampling-based approach to statistical parameter identification problems with random observations $\hat u$. By relating the observation noise in $\hat u$ to the uncertainty associated with the estimated parameter via Bayes' Theorem \cite{Stuart2010,Tarantola2005}, these methods allow us to sample directly from the joint distribution of $q$ at a given set of spatial points, through repeated evaluation of the deterministic forward model. The convergence of numerical implementations of Bayesian methods, most notably Markov chain Monte Carlo schemes, depends predominantly on the statistical complexity of the input $q$ and the measured output $\hat u$ and is often difficult to assess. In addition, the computational cost of evaluating the forward model can possibly severely limit their efficiency. 

\vspace{1em}
 
There has also been a continued interest in adapting variational methods  to estimate parameter uncertainty \cite{Banks2001, Rockafellar2007,Sandu2010, Zabaras2008}. Benefits include a well-established infrastructure of existing theory and algorithms, the possibility of incorporating multiple statistical influences, arising from uncertainty in boundary conditions or source terms for instance, and clearly defined convergence criteria. Let $(\Omega, \mathcal F_{\hat u}, d\omega)$ be a complete probability space and suppose we have a statistical model of the measured data $\hat u$ in the form of a random field $\hat u=\hat u(x,\omega)$ contained in the tensor product $\honeoinf(D):=H^1_0(D)\otimes L^2(\Omega)$. A least squares formulation of the parameter identification problem in \eqref{eqn: model}, when $q(x,\omega)$ is a random field, may take the form
\begin{equation}\label{eqn: infinite L2 problem}\tag{$P$}
\begin{split}
& \min_{(q,u)\in \htilde\times \honeoinf} J(q,u) :=\frac{1}{2}\|u-\hat u\|_{\honeoinf}^2+ \frac{\beta}{2} \|q\|_{\htilde}^2\\
s.t.\ \ & q \in \Qad, \ \ \ e(q,u) = 0,
\end{split}
\end{equation}
where the regularization term with $\beta > 0$ is added to ensure continuous dependence of the minimizer on the data $\hat u$. Here $\htilde := H(D)\otimes L^2(\Omega)$, where $H(D)$ is any Hilbert space that imbeds continuously in $L^\infty(D)$, which may be taken to be the Sobolev space $H^1(D)$ when $d=1$ or $H^2(D)$ when $d=2,3$ (see \cite{Ito1990a}). The feasible set $\Qad$ is given by
 \[\displaystyle \Qad = \{q\in \htilde(D): 0<\qmin\leq q(x,\omega) \hbox{\ a.s. on } D\times \Omega,\ \|q(\cdot,\omega)\|_H\leq \qmax \ \hbox{a.s. on } \Omega\},\]
while the stochastic equality constraint $e(q,u) = 0$ represents Equation \eqref{eqn: model}. It can also be written in its weak form as a functional equation $\tilde e(q,u) = 0$ in $\htilde^{-1}$, where
\begin{equation}\label{eqn: stochastic model}
\begin{split}
\langle \tilde e(q,u),v\rangle_{\htilde^{-1},\honeoinf} := &\int_\Omega \int_D q(x,\omega) \nabla u(x,\omega)\cdot \nabla v(x,\omega)\; dx\; d\omega -  \int_\Omega\int_D f(x)\phi(x,\omega)\; d\omega
\end{split}
\end{equation}
for all $v\in \honeoinf(D)$ \cite{Babuska2005}. For our purposes, it is useful to consider the equivalent functional equation $e(q,u) = 0$ in $\honeoinf$, where $e(q,u) := (-\Delta)^{-1}\tilde e(q,u)$ in the weak sense. Although these two forms of equality constraint are equivalent, pre-multiplication by the inverse Laplace operator adds a degree of preconditioning to the problem, as observed in \cite{Ito1990a}. We assume for the sake of simplicity that $f\in L^2(D)$ is deterministic.

\vspace{1em}

This formulation poses a number of theoretical, as well as computational challenges. The lack of smoothness of the random field $q=q(x,\omega)$ in its stochastic component $\omega$ limits the regularity of the equality constraint as a function of $q$, making it difficult to use theory analogous to the deterministic case in establishing first order necessary optimality conditions, as will be shown in Section \ref{section: infinite problem}. The most significant hurdle from a computational point of view is the need to approximate high dimensional integrals, both when evaluating the cost functional $J$ and when dealing with the equality constraint \eqref{eqn: stochastic model}. Monte Carlo type schemes seem inefficient, especially when compared with Bayesian methods. The recent success of Stochastic Galerkin methods \cite{Babuska2005, Xiu2002} and stochastic collocation-based approaches \cite{Babuska2005, Nobile2008} in efficiently estimating high dimensional integrals related to stochastic forward problems has, however, motivated investigations into their potential use in associated inverse and design problems.

\vspace{1em}

In forward simulations, collocation methods make use of spectral expansions, such as the Karhunen-Lo\`eve (KL) series, to approximate the known input random field $q$ by a smooth function of finitely many random variables, a so-called finite noise approximation. Standard PDE regularity theory \cite{Babuska2005} then ensures that the corresponding model output $u$ depends smoothly (even analytically) on these random variables. This facilitates the use of high-dimensional quadrature techniques, based on sparse grid interpolation of high order global polynomials. Inverse problems on the other hand are generally ill-posed and consequently any smoothness of a finite noise approximation of the given measured data $\hat u$ does not necessarily carry over to the unknown parameter $q$. In variational formulations, explicit assumptions should therefore be made on the smoothness of finite noise approximations of $q$ to facilitate efficient implementation, while also accurately estimating problem \eqref{eqn: infinite L2 problem}.

\vspace{1em}

We approximate \eqref{eqn: infinite L2 problem} in the space of functions with bounded mixed derivatives. Posing the finite noise minimization problem \eqref{eqn: finite noise problem} in this space not only guarantees that the equality constraint $e(q,u)$ is twice Fr\'echet differentiable in $q$ (see Section \ref{section: finite noise problem}), but also allows for the use of numerical discretization schemes based on sparse grid hierarchical finite elements, approximations known not only for their amenability to adaptive refinement, but also for their effectiveness in mitigating the curse of dimensionality \cite{Bungartz2004}. The authors in
\cite{Zabaras2008} demonstrate the use of piecewise linear hierarchical finite elements to approximate the finite noise design parameter in a least squares formulation of a heat flux control problem subject to system uncertainty, which is solved  numerically through gradient-based methods. This paper aims to provide a rigorous framework within which to analyze and numerically approximate problems of the form \eqref{eqn: infinite L2 problem}. 

\vspace{1em}

In Section \ref{section: infinite problem}, we establish existence and first order necessary optimality conditions for the infinite dimensional problem \eqref{eqn: infinite L2 problem}. In Section \ref{section: finite noise approximation} we make use of standard regularization theory to analytically justify the approximation of \eqref{eqn: infinite L2 problem} by the finite noise problem \eqref{eqn: finite noise problem}. We discuss existence and first order necessary optimality conditions for \eqref{eqn: finite noise problem} in Section \ref{section: finite noise problem} and formulate an augmented Lagrangian algorithm for finding its solution in Section \ref{section:augmented_lagrangian_alg}. Section \ref{section: discretization} covers the numerical approximation of $q$ and $u$, as well as the discretization of augmented Lagrangian optimization problem. Finally, we illustrate the application of our method on three numerical examples. 

\section{The Infinite Dimensional Problem}\label{section: infinite problem}
In order to accommodate the lack of smoothness of $q$ as a function of $\omega$ in our analysis, we impose inequality constraints uniformly in random space. Any function $q$ in the feasible set $\Qad$, satisfies the norm bound $\|q(\cdot, \omega)\|_H\leq \qmax$ uniformly on $\Omega$, which by the continuous imbedding of $H(D)$ into $L^\infty(D)$, implies $0<\qmin\leq q(x,\omega)\leq \qmax$ for all $(x,\omega)\in D\times \Omega$. This assumption, while ruling out unbounded processes, nevertheless reflects actual physical constraints. The uniform coercivity condition $0<\qmin\leq q(x,\omega)$, guarantees that for each $q\in \Qad$, there exists a unique solution $u=u(q)\in \honeoinf(D)$ to the weak form \eqref{eqn: stochastic model} of the equality constraint $e(q,u)=0$ \cite{Babuska2004} satisfying the bound
\begin{equation}\label{eqn: solution bound}\|u\|_{\honeoinf}^2 \leq \frac{C_D}{\qmin}\|f\|_{L^2}.\end{equation}
Hence all $q\in \Qad$ and their respective model outputs $u(q)$ have statistical moments of all orders. 
\subsection{Existence of Minimizers}
An explicit stability estimate of $u(q)$ in terms of the $L^p(D\times \Omega)$ norm of $q$ was given in \cite{Babuska2004, Babuska2005} for $2< p \leq \infty$. These norms, besides not having Hilbert space structure, give rise to topologies that are too weak for our purposes. The following lemmas establish the weak compactness of the feasible set, continuity of the solution mapping $q\mapsto u(q)$ restricted to $\Qad$, as well as the weak closedness of its graph in the stronger $\htilde$ norm and will be used to prove the existence of solutions to \eqref{eqn: infinite L2 problem}.
\begin{lemma}\label{lemma: Q convex and weakly compact}
The set $\Qad$ is closed, convex, and hence weakly compact in $\htilde$. 
\end{lemma}
\begin{proof}
Recall that 
\[\Qad=\{q\in \htilde: q(x,\omega)\geq \qmin,\ \hbox{ a.s. on $D\times \Omega$}, \ \|q(.,\omega)\|_{H}\leq \qmax \ \hbox{\ a.s. on $\Omega$}\}.
\]
Convexity is easily verified. To show that $\Qad$ is closed, let $\{q^n\}\subset \Qad$ and $q\in\htilde$ be such that  
\[\|q^n-q\|_{\htilde}^2= \int_\Omega \|q^n(.,\omega)- q(.,\omega)\|_H^2d\omega\rightarrow 0 \hspace{1cm} \hbox{as $n\rightarrow \infty$}.\] 
Since convergence in $L^2(\Omega,d\omega)$ implies pointwise almost sure convergence of a subsequence on $\Omega$, it follows that
\[\|q^{n_k}(.,\omega)-q(., \omega)\|_H\rightarrow 0 \ \ \ \hbox{a.s. on $\Omega$}\]
for some subsequence $\{q^{n_k}\}\subset \Qad$. Additionally, $\|q^{n_k}(.,\omega)\|_{H}\leq \qmax$ a.s.\;on $\Omega$ for $k\in \Natural$ and therefore $q$ also satisfies this constraint. Finally, $H(D)$ imbeds continuously in $L^\infty(D)$, which implies that the subsequence $\{q^{n_k}\}$ in fact converges to $q$ pointwise a.s.\;on $D\times \Omega$, ensuring that $q$ also satisfies pointwise constraint $q(x,\omega)\geq \qmin$ a.s.\;on $D\times \Omega$. 
\end{proof}

\begin{lemma}\label{lemma: u is continuous in q}
The mapping $u:q\in \Qad\mapsto u(q)\in \htilde^1_0$ is continuous. 
\end{lemma}
\begin{proof}
Suppose $q^n\rightarrow q$ in $\Qad$. As in the proof of the previous lemma, there exists a subsequence $q^{n_k}\rightarrow q$ pointwise a.s.\;on $D\times \Omega$. The upper bound on the function $u$ established in \cite[p.~1261]{Babuska2005} ensures that 
\[\|u(q^{n_k})-u(q)\|_{\honeoinf}\leq \left(\frac{C_D \|f\|_{L^2}}{\qmin^2}\right) \|q^{n_k}-q\|_{L^\infty(\Omega; L^\infty(D))}\rightarrow 0 \ \ \ \hbox{as $n\rightarrow \infty$},\]
where $C_D$ is the constant appearing in the Poincar\'e inequality on $D$. 
Furthermore, since any subsequence of $u(q^n)$ has a subsequence converging to $u(q)$, it follow that in fact $u(q^n)\rightarrow u(q)$. 
\end{proof}

\begin{lemma}\label{lemma: graph of u is weakly closed}
The graph $\{(q,u)\in \htilde\times \honeoinf: q\in \Qad, \ u=u(q)\}$ of $u$ is weakly closed. 
\end{lemma}
\begin{proof}
Let $q^n$ be a sequence in $\Qad$, so that $q^n\rightharpoonup q$ in $\htilde$ and $u(q^n)\rightharpoonup u$ in $\htilde^1_0$. The weak compactness of $\Qad$ shown in Lemma \ref{lemma: Q convex and weakly compact}, directly implies $q\in \Qad$. It now remains to be shown that $u=u(q)$ or equivalently that $u$ solves $e(q,u)=0$. Written in variational form, the requirement $e(q,u)=0$ is given by 
\begin{equation}\label{eqn: existence 1}
\int_\Omega \int_D q \nabla u\cdot \nabla v\; dx\; d\omega =  \int_\Omega\int_D f \; v\; dx\; d\omega \ \ \ \ \ \hbox{for all $v\in \honeoinf$}.
\end{equation}
Since the condition $u^n=u(q^n) \Leftrightarrow e(q^n,u^n)=0$ can be written as: 
\begin{equation}\label{eqn: existence 2}
\int_\Omega \int_D q^n \nabla u^n \cdot \nabla v\; dx\; d\omega = \int_\Omega\int_D f \; v\; dx\; d\omega  \ \ \ \ \ \hbox{for all $v\in\honeoinf$},
\end{equation}
it suffices to show that the left hand side of \eqref{eqn: existence 2} (or some subsequence thereof) converges to the left hand side of \eqref{eqn: existence 1} for all $v\in \honeoinf$. 
Now for any $n\geq 1$ and $v\in \honeoinf$, 
\begin{align*}
\int_\Omega \int_D (q^n \nabla u^n -q\nabla u) \cdot \nabla v \;dx\; d\omega &=\int_\Omega \int_D (q^n -q)\nabla u^n \cdot \nabla v \;dx\; d\omega\\
& + \int_\Omega \int_D q\nabla (u^n- u)\cdot \nabla v \;dx\; d\omega.\\
\end{align*}
Let $\{q^{n_k}\}$ be the subsequence of $\{q^n\}$ that converges to $q$ pointwise a.s.\;on $D\times\Omega$, as guaranteed by Lemma \ref{lemma: Q convex and weakly compact}. We can then bound the first term by
\begin{align*}
&\left |\int_\Omega \int_D (q^{n_k} -q)\nabla u^{n_k} \cdot \nabla v\; dx\; d\omega\right |\\
 \leq  &\left( \int_\Omega \int_D |q^{n_k}-q| |\nabla u^{n_k} |^2\; dx\; d\omega\right)^{\frac{1}{2}} \left( \int_\Omega \int_D |q^{n_k}-q|  |\nabla v|^2\; dx\; d\omega\right)^{\frac{1}{2}}\\
 \leq \  & 2\frac{\qmax}{\qmin}\|f\|_{L^2}\left( \int_\Omega \int_D |q^{n_k}-q|  |\nabla v|^2\; dx\; d\omega\right)^{\frac{1}{2}}\rightarrow 0\ \hbox{ as } n_k\rightarrow \infty,
\end{align*}
by the Dominated Convergence Theorem, since the integrand is bounded above by $2\qmax\|v\|_{\honeoinf}$. 

\vspace{1em}

\noindent The second term in this sum converges to $0$ due to the weak convergence $u^{n}\rightharpoonup u$ and the fact that the mapping 
$\|.\|_{q}: u \mapsto \|u\|_{q}:= \int_\Omega\int_D q|\nabla u|^2 dx d\omega$ defines a norm that is equivalent to $\|.\|_{\honeoinf}$, by virtue of the fact that $0<\qmin\leq q(x,\omega)\leq \qmax<\infty$.
Therefore
\begin{equation*}
\int_\Omega \int_D q \nabla u\cdot \nabla v \;dx\; d\omega =\lim_{n \rightarrow \infty}\int_\Omega \int_D q^n \nabla u^n\cdot \nabla v\; dx\; d\omega = \int_\Omega \int_D f v \;dx\; d\omega
\end{equation*}
for all $v\in \honeoinf$ and hence $e(q,u)=0$.
\end{proof}
By combining these lemmas, we can now show that a solution $q^*$ of the infinite dimensional minimization problem \eqref{eqn: infinite L2 problem} exists for any $\beta\geq 0$.
\begin{thm}[Existence of Minimizers]\label{thm: Existence of Min 1}
For each $\beta\geq 0$, the problem \eqref{eqn: infinite L2 problem} has a minimizer. 
\end{thm}
\begin{proof}
 Let $(q^n,u^n)$ be a minimizing sequence for the cost functional $J$ over $\Qad\times \honeoinf$, i.e.
\[ \inf_{(q,u)\in \Qad\times \honeoinf} J(q,u)=\lim_{n\rightarrow \infty} J(q^n,u^n)=\lim_{n\rightarrow \infty} \frac{1}{2} \|u^n-\hat u\|_{\honeoinf}^2+\frac{\beta}{2} \|q^n\|_{\htilde}^2\] 
Since $u^n$ satisfies the equality constraint $e(q^n,u^n) = 0$, and consequently $\|u^n\|_{\honeoinf}\leq \frac{1}{\qmin}\|f\|_{L^2}$ for all $n\geq 1$ (Lax-Milgram), the Banach Alaoglu theorem guarantees the existence of a weakly convergent subsequence $u^{n_k}\rightharpoonup u^* \in \honeoinf(D)$. Moreover, the weak compactness of $\Qad$ established in Lemma \ref{lemma: Q convex and weakly compact} also yields a subsequence $q^{n_k}\rightharpoonup q^*$ as $k\rightarrow \infty$, so that $q^*\in \Qad$. 
\noindent The fact that the infimum of $J$ is attained at the point $(q^*,u^*)$ follows directly from the weak lower semicontinuity of norms \cite{Reed1980}. Indeed,
\begin{align*}
J(q^*,u^*) &\leq \liminf_{n_1\rightarrow \infty} \frac{1}{2} \|u^{n_1}-\hat u\|_{\honeoinf}^2 + \liminf_{n_1\rightarrow \infty} \frac{\beta}{2} \|q^{n_1}\|_{\htilde}^2\\
&\leq \liminf_{n_1\rightarrow \infty} \left( \frac{1}{2} \|u^{n_1}-\hat u\|_{\honeoinf}^2 + \frac{\beta}{2} \|q^{n_1}\|_{\htilde}^2\right) = \inf_{(q,u)\in \Qad\times \honeoinf} J(q,u).
\end{align*}
Finally, it follows directly from Lemma \ref{lemma: graph of u is weakly closed} that $u^*=u^*(q^*)$ and hence $u^*$ satisfies the inequality constraint $e(q^*,u^*)=0$. The regularization term was not required to show the existence of minimizers.
\end{proof}

\subsection{A Saddle Point Condition}
Although solutions to \eqref{eqn: infinite L2 problem} exist, the inherent lack of smoothness of $q$ in the stochastic variable $\omega$ complicates the establishment of traditional necessary optimality conditions. A short calculation reveals that the equality constraint $e(q,u)=0$ is not Fr\'echet differentiable, as a function $q$ in $\htilde$. Additionally, the set of constraints has an empty interior in the $\htilde$-norm. Instead, we follow \cite{Chen1999} in deriving a saddle point condition for the optimizer $(q^*,u^*)$ of \eqref{eqn: infinite L2 problem} with the help of a Hahn-Banach separation argument.\\

Let $\langle \cdot, \cdot \rangle$ denote the $L^2(D\times \Omega)$ inner product. For any triple $(q,u,\lambda)\in \htilde \times \honeoinf \times \honeoinf$, we define the Lagrangian functional by
\begin{equation*}
L(q,u,\lambda) = J(q,u) + \langle e(q,u), \lambda \rangle_{\honeoinf} = \frac{1}{2}\|u-\hat u\|_{\honeoinf}^2 + \frac{\beta}{2} \|q\|_{\htilde}^2 + \langle q \nabla u, \nabla \lambda\rangle - \langle f, \lambda\rangle. 
\end{equation*}
The main theorem of this subsection is the following
\begin{thm}[Saddle Point Condition]
Let $(q^*,u^*)\in \Qad\times \honeoinf$ solve problem \eqref{eqn: infinite L2 problem}. Then there exists a Lagrange multiplier $\lambda^*\in \honeoinf$ so that the saddle point condition
\begin{equation}
L(q^*,u^*,\mu)\leq L(q^*,u^*,\lambda^*)\leq L(q,u,\lambda^*)
\end{equation}
holds for all  $(q,u,\mu)\in \Qad\times \honeoinf \times \honeoinf$.
\end{thm}
\begin{proof}
Note that the second inequality simply reflects the optimality of $(q^*,u^*)$. To obtain the first inequality, we rely on a Hahn-Banach separation argument. Let 
\[S=\{(J(q,u)-J(q^*,u^*)+s, e(q,u))\in \mathbb R \times \honeoinf:  (q,u)\in \Qad\times \honeoinf, s\geq 0\}\]
and
\[T=\{(-t,0)\in \real \times \honeoinf: t>0\}\]
In the ensuing three lemmas we will show that
\begin{enumerate}
\item $S$ and $T$ are convex (Lemma \ref{lemma: S and T convex}), 
\item $S\cap T = \emptyset$ (Lemma \ref{lemma: S and T disjoint}),  and 
\item $S$ has at least one interior point (Lemma \ref{lemma: S has non-empty interior}).
\end{enumerate}
The Hahn-Banach Theorem thus gives rise to a separating hyperplane, i.e. a pair $(\alpha_0, \lambda_0)\neq (0,0)$ in $\R \times \honeoinf$, such that 
\begin{equation}\label{eqn: Hahn-Banach separation}
\alpha_0 (J(q,u)-J(q^*,u^*)+s) + \langle e(q,u), \lambda_0 \rangle_{\honeoinf} \geq -t\alpha_0  \ \ \ \forall t>0, s\geq 0, (q,u)\in \Qad\times \honeoinf.
\end{equation} 
Letting $s=t=1$ and $(q,u)=(q^*,u^*)$ readily yields $\alpha_0\geq 0$. In fact $\alpha_0 >0$. Suppose to the contrary that $\alpha_0 = 0$. Then by \eqref{eqn: Hahn-Banach separation}
\[\langle e(q,u), \lambda_0 \rangle_{\honeoinf} = \langle q \nabla u, \nabla \lambda_0\rangle -\langle f, \lambda_0\rangle\geq 0\ \ \ \forall (q,u)\in \Qad\times \honeoinf\]
particularly for $q=q^*$ and $u\in \honeoinf$ satisfying $\langle q^*\nabla u, \nabla \phi \rangle -\langle f-\lambda_0, \phi \rangle = 0 \ \forall \phi\in \honeoinf$, we have
\[ \langle q^* \nabla u, \nabla \lambda_0 \rangle -\langle f, \lambda_0\rangle = \langle f-\lambda_0,\lambda_0 \rangle - \langle f,\lambda_0\rangle = -\langle \lambda_0,\lambda_0\rangle\geq 0,\]
which implies that $\lambda_0=0$. This contradicts the fact that $(\alpha_0, \lambda_0)\neq (0,0)$. Dividing \eqref{eqn: Hahn-Banach separation} by $\alpha_0$ and letting $\lambda^* = \lambda_0/\alpha_0$ yields $J(q^*,u^*)\leq J(q,u) + \langle e(q,u), \lambda^* \rangle_{\honeoinf}\ \forall (q,u)\in \Qad\times \honeoinf$ and hence 
\begin{align*}
L(q^*,u^*,\mu) &= J(q^*,u^*)+ \langle e(q^*,u^*), \mu \rangle_{\honeoinf} = J(q^*,u^*)\\
&\leq J(q,u)+ \langle e(q,u), \lambda^* \rangle_{\honeoinf} = L(q,u, \lambda^*)
\end{align*}
for all $(q,u,\mu)\in \Qad\times\honeoinf\times \honeoinf$.
\end{proof}

\begin{lemma}\label{lemma: S and T convex}
The sets $S$ and $T$ are convex. 
\end{lemma}
\begin{proof}
Clearly, $T$ is convex. Let $0\leq  \alpha\leq 1$ and consider the convex combination $P_\alpha = \alpha P_1 +(1-\alpha) P_2$ where $P_1,P_2\in S$. Hence $P_\alpha$ is of the form $P_\alpha = (p_\alpha,w_\alpha)$ where 
\begin{align*}
p_\alpha &= \alpha(J(q_1,u_1)-J(q^*,u^*)+s_1)+(1-\alpha)(J(q_2,u_2)-J(q^*,u^*)+s_2)\\
w_\alpha&=\alpha e(q_1,u_1)+(1-\alpha)e(q_2,u_2)
\end{align*}
with $q_1,q_2\in \Qad,\ u_1,u_2\in \honeoinf, \ \hbox{and} \ s_1,s_2 \geq 0$. It now remains to show that $w_\alpha= e(q_\alpha,u_\alpha)$ for some $(q_\alpha,u_\alpha)\in \Qad\times \honeoinf$ and $p_\alpha = J(q_\alpha,u_\alpha)-J(q^*,u^*)+s_\alpha$ for some $s_\alpha\geq 0$. Let $q_\alpha= \alpha q_1+(1-\alpha) q_2\in \Qad$ and let $u_\alpha\in \honeoinf$ be the unique solution of the variational problem
\[\langle q_\alpha \nabla u_\alpha, \nabla \phi\rangle = \langle \alpha q_1 \nabla u_1 + (1-\alpha) q_2 \nabla u_2, \nabla \phi \rangle\hspace{1cm} \forall \phi \in \honeoinf.\]
Therefore
\begin{align*}
\langle w_\alpha, \phi \rangle_{\honeoinf} &= \langle \alpha q_1 \nabla u_1 +(1-\alpha)q_2\nabla u_2, \nabla \phi \rangle -\langle f,\phi \rangle\\
&= \langle q_\alpha \nabla u_\alpha, \nabla \phi \rangle - \langle f,\phi \rangle = \langle e(q_\alpha,u_\alpha), \phi \rangle_{\honeoinf} &\forall \phi\in \honeoinf
\end{align*}
which implies that $w_\alpha = e(q_\alpha, u_\alpha)$.  Moreover, it follows readily from the convexity of norms that 
\[J(q_\alpha,u_\alpha)\leq \alpha J(q_1,u_2)+(1-\alpha) J(q_2,u_2)\]
and therefore letting 
\begin{align*}
s_\alpha& =  \alpha J(q_1,u_1)+(1-\alpha)J(q_2,u_2) - J(q_\alpha,u_\alpha)+\alpha s_1 + (1-\alpha)s_2 \geq  \alpha s_1 + (1-\alpha)s_2 \geq 0
\end{align*}
we obtain
\[p_\alpha = J(q_\alpha,u_\alpha) - J(q^*,u^*) +s_\alpha.\]
\end{proof}

\begin{lemma}\label{lemma: S and T disjoint}
The sets $S$ and $T$ are disjoint. 
\end{lemma}
\begin{proof}
This follows directly from the fact that $J(q,u)\geq J(q^*,u^*)$ for all points $(q,u)$ in $\Qad\times \honeoinf $
\end{proof} 

\begin{lemma}\label{lemma: S has non-empty interior}
The set $S$ has a non-empty interior.
\end{lemma}
\begin{proof}
Clearly $(s_0, 0)= (J(q^*,u^*)-J(q^*,u^*)+s_0,e(q^*,u^*))\in S$ for any $s_0 > 0$.  For any $\epsilon \in (0,1)$, let $(s,w)$ belong to the $\epsilon$-neighborhood of $(s_0,0)$. In other words $|s-s_0|+\|w\|_{\honeoinf} \leq \epsilon$. Let $q=q^*$ and let $u$ be the solution to the problem 
\begin{equation}\label{eqn: non-empty interior pde}
\langle q^* \nabla u, \nabla \phi \rangle = \langle f,\phi\rangle + \langle\nabla w, \nabla \phi\rangle \hspace{2cm} \forall \phi \in \honeoinf(D).
\end{equation}
Clearly, $w = e(q^*,u)$ by definition. Then
\begin{align*}
s' &:= s_0 + J(q^*,u^*)- J(q,u) = s_0 + J(q^*,u^*) - J(q^*,u)\\
&=s_0 + \frac{1}{2}\int_\Omega \int_D |\nabla (u^*(x,\omega)-\hat u(x,\omega))|^2 dx\;d\omega -   \frac{1}{2}\int_\Omega \int_D |\nabla (u(x,\omega)-\hat u(x,\omega))|^2 dx\;d\omega \\
&=s_0 -  \frac{1}{2}\int_\Omega \int_D \nabla (u(x,\omega)-u^*(x,\omega)) \cdot \nabla(u(x,\omega)+u^*(x,\omega)-2\hat u(x,\omega)) \;dx\; d\omega
\end{align*}
Now $u^*$ satisfies $e(q^*,u^*)=0$ and hence $\|u^*\|_{\honeoinf}\leq \frac{C_D}{\qmin} \|f\|_{L^2}$ by \eqref{eqn: solution bound}. Similarly, since $u$ solves \eqref{eqn: non-empty interior pde}, it follows that $\|u\|_{\honeoinf} \leq \frac{C_D}{\qmin}(\|f\|_{L^2}+\|w\|_{\honeoinf}) \leq \frac{C_D}{\qmin}(\|f\|_{L^2}+\epsilon)$ and hence $\|u-u^*\|_{\honeoinf}\leq \frac{C_D}{\qmin}\epsilon$. We therefore have
\begin{align*}
s'&\geq s_0- \frac{1}{2} \|u-u^*\|_{\honeoinf} (\|u^*\|_{\honeoinf}+\|u\|_{\honeoinf}+2\|\hat u\|_{\honeoinf})\\
&\geq  s_0- \frac{\epsilon}{2 \qmin}(\frac{C_D}{\qmin} \|f\|_{L^2} + \frac{C_D}{\qmin}(\|f\|_{L^2}+\epsilon) + 2\|\hat u\|_{\honeoinf} )\\
&\geq s_0 - \frac{\epsilon}{2\qmin^2}(2C_D\|f\|_{L^2} +C_D\epsilon + 2\qmin \|\hat u\|_{\honeoinf} )\geq 0
\end{align*}
for small enough $\epsilon>0$. Therefore $(s,w)= (J(q^*,u)-J(q^*,u^*)+s', e(q^*,u))\in S$ for any $(s,w)$ in a small enough $\epsilon$-neighborhood of $(s_0,0)$. 
\end{proof}

In the following section, we will show that if the observed data $\hat u$ is expressed as a Karhunen-Lo\`eve series \cite{Loeve1978, Schwab2006}, we may approximate problem \eqref{eqn: infinite L2 problem} by a finite noise optimization problem \eqref{eqn: finite noise problem}, where $q$ is  a smooth, albeit high-dimensional, function of $x$ and intermediary random variables $\{Y_i\}_{i=1}^n$. The convergence framework not only informs the choice of numerical discretization, but also suggests the use of a dimension-adaptive scheme to exploit the progressive \lq smoothing' of the problem. 

\section{Approximation by the Finite Noise Problem}\label{section: finite noise approximation}
According to \cite{Loeve1978}, the random field $\uhat$ may be written as the Karhunen-Lo\`eve (KL) series
\begin{equation}
\uhat(x,\omega) = \uhat_0(x) + \sum_{k=1}^\infty \sqrt{\nu_k}b_k(x) Y_k(\omega),
\end{equation}
where $\{Y_k\}_{k=1}^\infty$ is an uncorrelated orthonormal sequence of random variables with zero mean and unit variance and $(\nu_k, b_k)$ is the eigenpair sequence of $\uhat$'s compact covariance operator $\mathscr C_{\uhat}: H^1_0(D)\rightarrow H^1_0(D)$ \cite{Schwab2006}. Moreover, the truncated series 
\begin{equation*}
\hat u^n(x,\omega) = u_0(x) + \sum_{k=1}^n \sqrt{\nu_k}b_k(x)Y_k(\omega)
\end{equation*}
converges to $\uhat$ in $\honeoinf$, i.e. $\|\uhat- \uhat^n\|_{\honeoinf}\rightarrow 0$ as $n\rightarrow \infty$. Assume w.l.o.g. that $\{Y_i\}_{i=1}^\infty$ forms a complete orthonormal basis for $L^2_0(\Omega)$, the set of functions in $L^2(\Omega)$ with zero mean. If this is not the case, we can restrict ourselves to $L_0^2(\Omega)\cap \overline{\Span\{Y_i\}}$. The following additional assumption imposes restrictions on the range of the random vectors we consider.
\begin{assumption}\label{ass: random variables bounded}
Assume the random variables $\{Y_n\}$ are bounded uniformly in $n$, i.e. 
\[y_{\mathrm{min}}\leq Y_n(\omega) \leq y_{\mathrm{max}} \ \ \ \hbox{a.s.\;on $\Omega$ for all $n\in\Natural$ and some $y_{\mathrm{min}}, y_{\mathrm{max}} \in \R$.}\]
 Furthermore, assume that for any $n$ the probability measure of the random vector $Y =(Y_1,...,Y_n)$ is absolutely continuous with respect to the Lebesgue measure and hence $Y$ has joint density $\rho_n:\Gamma^n\rightarrow [0, \infty)$, where the hypercube $\Gamma^n:=\prod_{i=1}^n\Gamma_i\subset [y_{\mathrm{min}}, y_{\mathrm{max}}]^n$ denotes the range of $Y$. 
\end{assumption}

Since $\hat u^n$ depends on $\omega$ only through the intermediary variables $\{Y_i\}_{i=1}^n$, it seems reasonable to also estimate the unknown parameter $q^n$ as a function of these, i.e. 
\[
q^n(x,\omega) = q^n(x, Y_1(\omega),...,Y_n(\omega)).
\]
The appropriate parameter space for the finite noise identification problem is not immediately apparent. In order for the finite noise optimization problem to approximate \eqref{eqn: infinite L2 problem}, $q_n$ should at the very least be square integrable in $y$, i.e. $q^n \in \htildefin(D):=H(D)\otimes L^2(\Gamma^n)\subset \htilde(D)$. With this parameter space, however, the finite noise problem suffers from the same lack of regularity  encountered in the infinite dimensional problem \eqref{eqn: infinite L2 problem}. In order to ensure both that the finite noise equality constraint $e_n(q,u)=0$ is Fr\'echet differentiable and that the set of admissible parameters $\Qad$ has a non-empty interior, we require a higher degree of smoothness in $q$ as a function of $y\in \Gamma^n$. 

\vspace{1em}

For the sake of our analysis, we therefore seek finite noise minimizers $q_n^*$ in the space $\hmix:=H(D)\otimes H^s_{\mathrm{mix}}(\Gamma^n)$, where $H_{\mathrm{mix}}^s(\Gamma^n)$ is the space of functions with bounded mixed derivatives, $s\geq 1$ \cite{Temlyakov1993}. A function $v\in \hmix\subset L^2(D\times \Gamma^n)$ is one for which the $\hmix$-norm, 
\begin{equation}\label{eqn: hmix norm formula}
\|v\|_{\hmix}^2 := \sum_{|\gamma|_{\infty}\leq s}\sum_{|\alpha|_1\leq t_d} \int_D \int_{\Gamma^n} \left | D_y^\gamma D_x^\alpha v(x,y)\right |^2 \rho_n(y) dy dx 
\end{equation}
is finite, where $\gamma = (\gamma_1,..., \gamma_n)\in \mathbb N^n$   and $\alpha = (\alpha_1,...,\alpha_d)\in \mathbb N^d$ are multi-indices, with $|\gamma|_{\infty} = \max\{\gamma_1,...,\gamma_n\}$, $|\alpha|_1=\alpha_1+...+\alpha_n$ and $t_d= 1$ when $d=1$ or $t_d = 2$ when $d=2,3$. Apart from considerations of convenience, the use of this parameter space is partly justified by the fact that $\{Y_n\}_{n=1}^\infty$ forms a basis for $L_0^2(\Omega)$. The minimizer $q^*$ of the original infinite dimensional problem \eqref{eqn: infinite L2 problem} thus takes the form
\begin{equation*}
q^*(x,\omega) = q_0^*(x) + \sum_{n=1}^\infty q_n(x) Y_n(\omega),
\end{equation*}
which is linear in each of the random variables $Y_n$. Any minimizer $q_n^*$ of \eqref{eqn: finite noise problem} that approximates $q^*$ (even in the weak sense) is therefore expected to depend relatively smoothly on $y$ when $n$ is large. At low orders of approximation, on the other hand, the parameter $q$ that gives rise to the model output $u(q)$ most closely resembling the partial data $\hat u^n$ may not exhibit the same degree of smoothness in the variable $y=(y_1,...,y_n)$. Since the accuracy in approximation of functions in high dimensions benefits greatly from a high degree of smoothness \cite{Barthelmann2000}, this suggests the use of a dimension adaptive strategy in which the smoothness requirement of the parameter is gradually strengthened as the stochastic dimension $n$ increases.

\vspace{1em}

We can now proceed to formulate a finite noise least squares parameter estimation problem for the perturbed, finite noise data $\hat u^n$:
\begin{equation}\label{eqn: finite noise problem}\tag{$P^n$}
\begin{split}
\min_{(q,u) \in \hmix\times \honeofin} J(q,u) &:=\frac{1}{2}\|u-\hat u^n\|_{\honeofin}^2 + \frac{\beta_n}{2}\|q\|_{\hmix}^2\\
s.t.\ \  q\in\ & \Qad^n, \ \ e_n(q,u) = 0 
\end{split}
\end{equation}
where $e_n(\cdot,\cdot):\hmix\times\honeofin \rightarrow \honeofin$ is defined by $e_n(q,u) = (-\Delta)^{-1}\tilde e_n(q,u)$ with 
\begin{align*}
\langle \tilde e_n(q,u),v\rangle_{\htildefin^{-1},\honeofin} := & \int_{\Gamma^n}\int_D q(x,y)\nabla u(x,y)\cdot \nabla v(x,y) \rho_n(y)\; dx\; dy \\
& - \int_{\Gamma^n} f(x,y) v(x,y) \rho_n(y) \;dx\;dy
\end{align*}
for all $v\in \honeofin(D)$, and
\begin{equation*}
\Qad^n:=\left\{ q\in \widetilde H^n: \begin{array}{ll} \ 0<\qmin-\frac{1}{k_n}\leq q(x,y) & \hbox{\ a.s. on } D\times \Gamma^n,\\
\|q(\cdot,y)\|_H\leq \qmax +\frac{1}{k_n}\ &\ \hbox{a.s. on } \Gamma^n
\end{array}\right\}.
\end{equation*}
with $k_n\rightarrow \infty$ a monotone increasing approximation parameter to be specified later.

\vspace{1em}

In the following, we justify the use of this approximation scheme by demonstrating that it not only lends itself more readily to standard first- and second-order optimization theory, but also that \eqref{eqn: finite noise problem} approximates \eqref{eqn: infinite L2 problem} in a certain sense. In particular, we first show that, as $n\rightarrow \infty$ and $\beta_n\rightarrow 0$, the sequence of minimizers $q_n^*$ of problem \eqref{eqn: finite noise problem} has a weakly convergent subsequence and that the limits of all convergent subsequences minimize the infinite dimensional problem \eqref{eqn: infinite L2 problem}. Tikhonov regularization theory for non-linear least squares problems \cite{Binder1994} provides the theoretical framework underlying the arguments in this section.  

\vspace{1em}

In order to mediate between the minimizer $q_n^*$ of the finite noise problem \eqref{eqn: finite noise problem}, formulated in the $\hmix$ norm, and that of the infinite dimensional problem, whose minimizer $q^*$ is measured in the $\htilde$ norm, we make use of the projection of $q^*$ on the first $n$ basis vectors:
\[\proj^n q^*=q_0^*(x) + \sum_{i=1}^n q_i(x) Y_i(\omega).\]
Evidently, $\proj^n q^* \rightarrow q^*$ as $n\rightarrow \infty$ in $\htilde$. Moreover, seeing that $\proj^n q^*$ is linear in $y$, it's norm in $\hmix$ can be bounded in terms of its norm in $\htilde$ as the following computation shows: 
\begin{lemma}
\[\|\proj^n q^*\|_{\hmix} \leq \sqrt 2 \|\proj^n q^*\|_{\htilde}. \]
\end{lemma}
\begin{proof}
Let $e_i$ be the $i^{th}$ standard basis vector for $\Natural^n$. We now apply expression \eqref{eqn: hmix norm formula} to $\proj^n q^*$ to obtain
\begin{align*}
&\|\proj^n q^*\|_{\hmix}^2:= \sum_{|\gamma|_\infty\leq s}\sum_{|\alpha|_1\leq t_d} \int_D \int_{\Gamma^n} \left | D_y^{\gamma} D_x^\alpha \left[q_0(x) + \sum_{i=1}^n q_i(x) y_i\right] \right |^2 \rho_n(y) dy dx \\
&= \sum_{|\alpha|_1\leq t_d} \int_D \int_{\Gamma^n} \left | D_y^{0} D_x^\alpha \left[q_0(x) + \sum_{i=1}^n q_i(x) y_i\right] \right |^2 \rho_n(y) dy dx\\
&+ \sum_{i=1}^n \sum_{|\alpha|_1\leq t_d} \int_D \int_{\Gamma^n} \left | D_y^{e_i} D_x^\alpha \left[\sum_{i=1}^n q_i(x) y_i\right] \right |^2 \rho_n(y) dy dx\\
&=\int_{\Gamma^n}\|\proj^n q^*(\cdot, \omega)\|_{H}^2\rho_n(y)dy + \sum_{i=1}^n \sum_{|\alpha|_1\leq t_d} \int_D \int_{\Gamma^n} \left | D_x^\alpha q_i(x) \right |^2 \rho_n(y) dy dx\\
&=  \|\proj^n q^*\|_{\tilde H}^2 +\sum_{i=1}^n \|q_i\|_{H}^2 = 2\sum_{i=0}^n \|q_i\|_{H}^2-\|q_0\|_H^2 \leq 2 \|\proj^n q^*\|_{\htilde}^2.\\
\end{align*}
The second and third equalities follow from the fact that
\begin{equation*}
D_y^\gamma \left[\sum_{i=1}^nq_i(x) y_i \right]=
\left\{\begin{array}{ll}
\sum_{i=1}^nq_i(x) y_i &,\ \hbox{if } \gamma =0\\ 
q_i(x) &,\ \hbox{if } \gamma = e_i\\
0 &,\ \hbox{otherwise}
\end{array}\right. .
\end{equation*}
\end{proof}
The next lemma addresses the feasibility of $\proj^n q^*$.  Although $\proj^n q^* $ does not necessarily lie in the feasible region $\Qad$, the set on which $\proj^n q^*\notin \Qad$ can be made arbitrarily small as $n\rightarrow\infty$. Let $\mathscr A_n$ be the event that $\proj^n q^*$ lies inside the approximate feasible region $\Qad^n$, i.e.
\begin{align*}
\An &:= \{\omega\in \Omega: 0 < \qmin - \frac{1}{k_n}\leq \proj^n q^*(x,\omega)\ a.s.\ \hbox{on } D,\ \|\proj^n q^*(\cdot, \omega)\|_H\leq \qmax+\frac{1}{k_n}\ \}.\\
\end{align*}
Then we have
\begin{lemma}\label{lemma: projection is almost feasible}
There is a monotonically increasing sequence $k_n\rightarrow \infty$ so that $\prob(\Omega\backslash\An)\leq \frac{1}{k_n}$ for all $n\in\Natural$. 
\end{lemma} 
\begin{proof}
For any $n\geq 1$, let $k_n$ satisfy $\|\proj^n q^* -q^*\|_{\htilde}^2= \frac{1}{C^2k_n^3}$, where $C\geq 1$ is the imbedding constant for $H(D)\hookrightarrow L^\infty(D)$. Clearly $k_n\rightarrow \infty$ as $n\rightarrow \infty$. Let 
\[\mathscr B_n = \{\omega\in\Omega: \|\proj^nq^*(\cdot,\omega)-q^*(\cdot,\omega)\|_{H}\leq \frac{1}{C k_n}\}.\]
For any $\omega \in \mathscr B_n$, 
\begin{align*}
\Big|\|\proj^nq^*(\cdot,\omega)\|_H-\|q^*(\cdot, \omega)\|_H\Big|\leq \|\proj^n q^*(\cdot,\omega) -q^*(\cdot, \omega)\|_H\leq \frac{1}{Ck_n}\leq \frac{1}{k_n}
\end{align*}
and 
\begin{align*}
\|\proj^n q^*(\cdot,\omega) -q^*(\cdot, \omega)\|_{L^\infty}\leq C\|\proj^n q^*(\cdot,\omega) -q^*(\cdot, \omega)\|_H \leq \frac{1}{k_n},
\end{align*}
which implies $\mathscr B_n\subset \An$. Moreover, according to Chebychev's inequality
\begin{align*}
\prob(\Omega\backslash\An)\leq \prob(\Omega\backslash\mathscr B_n)\leq C^2k_n^2 \int_{\Omega} \|\proj^n q^*(\cdot,\omega) -q^*(\cdot, \omega)\|_H^2 d\omega = C^2k_n^2\|\proj^n q^* -q^*\|_{\htilde}^2\leq \frac{1}{k_n}.
\end{align*}
\end{proof}
In order to ensure strict adherence to the inequality constraints of \eqref{eqn: finite noise problem} for every $n$, we modify $\proj^n q^*(\cdot, \omega)$ on $\Omega\backslash\An$.
\begin{definition}
For all $n\in \mathbb N$, let $\qhatn\in \hmix\subset \htilde$ be defined as follows: 
\begin{equation}
\qhatn := \left\{ \begin{array}{cc} \proj^n q^*, &\  \omega \in \An\\ q_n^*, &\  \omega \notin \An \end{array} \right. . 
\end{equation}
\end{definition}
\noindent Evidently $\qhatn\in \Qad\cap \hmix$ and in light of Lemma \ref{lemma: projection is almost feasible}, it is reasonable to expect $\qhatn\approx \proj^n q^*$ for large $n$, except on sets of negligible measure. Indeed
\begin{lemma}\label{lemma: qhat approximates q*}
 $\displaystyle \qhatn\rightarrow q^*$ in $\htilde$ as $n\rightarrow \infty$.
\end{lemma}
\begin{proof}
\begin{align*}
\|\qhatn-q^*\|_{\htilde} &= \int_{\An}\|\proj^nq^*(\cdot, \omega)-q^*(\cdot, \omega)\|_{H}^2 d\omega + \int_{\Omega\backslash\An}\|q_n^*(\cdot, \omega)-q^*(\cdot, \omega)\|_{H}^2 d\omega \\
&\leq \|\proj^n q^*-q^*\|_{\htilde}^2 + \prob(\Omega\backslash\An)\sup_{\omega\in \Omega}\|q_n^*(\cdot, \omega)-q^*(\cdot, \omega)\|_{H}^2 \\
&\leq \|\proj^n q^*-q^*\|_{\htilde}^2 + \frac{1}{k_n}4(\qmax+\frac{1}{k_1})^2 \rightarrow 0.
\end{align*}
\end{proof}
\noindent We are now in a position to prove the main theorem of this section. For its proof we will make use of the fact that, due to the lower semicontinuity of norms
\begin{equation}\label{eqn: weak conv and limsup bounded imply strong conv}
x_n \rightharpoonup x, \ \ \limsup_{n\rightarrow \infty} \|x_n\|\leq \|x\| \ \Rightarrow \ x_n \rightarrow x
\end{equation}
for any sequence $x_n$ in a Hilbert space.
\begin{thm}
Let $\|\hat u-\hat u^n\|_{\htilde^1_0}\rightarrow 0$ and $\beta_n\rightarrow 0$ as $n\rightarrow \infty$. Then the sequence of minimizers $q_n^*$ of \eqref{eqn: finite noise problem} has a subsequence converging weakly to a minimizer of infinite dimensional problem \eqref{eqn: infinite L2 problem} and the limit of every weakly convergent subsequence is a minimizer of \eqref{eqn: infinite L2 problem}. The corresponding model outputs converge strongly to the infinite dimensional minimizer's model output. 
\end{thm}
\begin{proof}
Since $q_n^*$  is optimal for \eqref{eqn: finite noise problem}, we have
\begin{equation}\label{eqn: discrete optimality}
\|u(q_n^*) - \hat u^n\|_{\honeoinf}^2 + \beta_n \|q_n^*\|_{\hmix}^2 \leq \|u(\qhatn) - \hat u^n\|_{\honeoinf}^2 + \beta_n \|\qhatn\|_{\hmix}^2.
\end{equation}
Moreover, by definition $\qhatn(\cdot, Y(\omega))=q_n^*(\cdot, Y(\omega))$ for all $Y\in Y(\Omega\backslash\An)$ and hence 
\begin{align*}
&\|\qhatn\|_{\htilde}^2 - \| q_n^*\|_{\htilde}^2 \\
= & \sum_{|\gamma|_\infty\leq 1}\sum_{|\alpha|_1\leq t_d}\left( \int_{Y(\An)}\int_D \left | D_y^{\gamma} D_x^\alpha  \proj^n q^*\right |^2 \rho_n(y) dxdy - \int_{Y(\An)}\int_D \left | D_y^{\gamma} D_x^\alpha  q_n^*\right |^2 \rho_n(y) dx dy \right)\\
\leq & \sum_{|\gamma|_\infty\leq 1}\sum_{|\alpha|_1\leq t_d} \left(\int_{Y(\An)}\int_D \left | D_y^{\gamma} D_x^\alpha  \proj^n q^*\right |^2 \rho_n(y) dx dy \right)\leq \|\proj^nq^*\|_{\hmix}^2  \leq 2 \|\proj^n q^*\|_{\tilde H}^2
\end{align*}
from which it follows that
\begin{align*}
\|u(q_n^*) - \hat u^n\|_{\honeoinf}^2 &\leq \|u(\hat q_n^*) - \hat u^n\|_{\honeoinf}^2 + \beta_n \|\hat q_n^*\|_{\hmix}^2 -\beta_n \|q_n^*\|_{\hmix}^2 \\
& \leq \|u(\hat q_n^*) - \hat u^n\|_{\honeoinf}^2 + \beta_n \|\proj^n q^*\|_{\htilde}^2.
\end{align*}
By Lemmas \ref{lemma: qhat approximates q*} and \ref{lemma: u is continuous in q}
\begin{align*}
\limsup_{n\rightarrow \infty} \|u(q_n^*) - \hat u^n\|_{\honeoinf}^2 &\leq \lim_{n\rightarrow \infty}  \|u(\qhatn) - \hat u^n\|_{\htilde^1_0}^2 + \beta_n \|\proj^n q^*\|_{\htilde}^2 = \|u(q^*) - \hat u\|_{\htilde^1_0}^2,
\end{align*}
which, together with the Banach Alaoglu Theorem, guarantees the existence of  a subsequence $u(q_{n_j}^*)$ converging weakly to some $u_0\in\htilde^1_0$. Since feasible sets $\{\Qad^n\}_{n=1}^\infty$ form a nested sequence, all  functions $q_{n}^*\in \Qad^n\subset \Qad^1$, which is weakly compact (Lemma \ref{lemma: Q convex and weakly compact}). The sequence $q_n^*\in \Qad$ therefore has a subsequence, $q_{n_j}^* \rightharpoonup q_0\in \Qad^1$ in $\htilde$. Additionally, since $\Qad^n$ is nested and the graph of $u$ is weakly closed (Lemma \ref{lemma: graph of u is weakly closed}) we have $q_0\in \cap_{n=1}^\infty \Qad^n = \Qad$ and $u_0 = u(q_0)$. Therefore
\begin{align}
\|u(q_0)-\hat u\|_{\honeoinf}^2 &= \lim_{j\rightarrow \infty}\langle u(q_{n_j}^*)-\hat u^{n_j}, u(q_0)-\hat u \rangle_{\honeoinf} \nonumber\\
&\leq \liminf_{j\rightarrow \infty} \|u(q_{n_j}^*)-\hat u^{n_j}\|_{\honeoinf}\|u(q_0)-\hat u\|_{\honeoinf}\label{eqn: weak limit u less than liminf u}\\
&\leq \limsup_{j\rightarrow \infty} \|u(q_{n_j}^*)-\hat u^{n_j}\|_{\honeoinf}\|u(q_0)-\hat u\|_{\honeoinf}\label{eqn: liminf u less than limsup u}\\
&\leq \|u(q^*) - \hat u\|_{\honeoinf} \|u(q_0)-\hat u\|_{\honeoinf}\nonumber,
\end{align}
which implies 
$\|u(q_0)-\hat u\|_{\htilde^1_0}\leq \|u(q^*) - \hat u\|_{\htilde^1_0}$ and hence $q_0\in \Qad$ is a minimizer for $\eqref{eqn: infinite L2 problem}$. Inequalities \eqref{eqn: weak limit u less than liminf u} and \eqref{eqn: liminf u less than limsup u} further imply
\[\lim_{j\rightarrow \infty} \|u(q_{n_j}^*)-\hat u^{n_j}\|_{\honeoinf} = \|u(q_0)-\hat u\|_{\honeoinf},\]
which, together with the weak convergence $u(q_{n_j}^*)-\hat u^{n_j}\rightharpoonup u(q_0)-\hat u$, implies $u(q_{n_j}^*)-\hat u^{n_j}\rightarrow u(q_0)-\hat u$
due to \eqref{eqn: weak conv and limsup bounded imply strong conv}. In addition, the fact that $\hat u^{n_j}\rightarrow \hat u$ implies that $u(q_{n_j})\rightarrow u(q_0)$. Finally, this argument holds for any convergent subsequence of $\{q_n^*\}$ and hence the Theorem is proved.
\end{proof}

\section{The Finite Noise Problem}\label{section: finite noise problem}
The immediate benefit of using $\hmix$ as an approximate search space is that it imbeds continuously in $L^\infty(D\times \Gamma^n)$, regardless of the size of the stochastic dimension $n$. By virtue of the tensor product structure of $\hmix(\Gamma^n)$ we may consider Sobolev regularity component-wise, which, in conjunction with the compact imbedding of $H^1(\Gamma_i)$ in $L^\infty(\Gamma_i)$, gives rise to this property as the following lemma shows. 
\begin{lemma}\label{lemma: hmix imbeds continuously in Linfty}
The space $\hmix$ imbeds continuously in $L^{\infty}(D\times \Gamma^n)$ for all $n\in\Natural$.
\end{lemma}
\begin{proof}
For any fixed value $y_0$ of the random component $y$ and any multi-index $\gamma\in \mathbb N^{n}$, the function $D^{\gamma}_{y} q(\cdot, y_0)\in H^{t_d}(D)$ whenever $|\gamma|_{\infty}\leq s$. Similarly, if both spatial variable $x$ and all but the $i^{th}$ component $y_i$ of the stochastic variable $y$ are fixed at $x_0$ and $y_0^1,...,y_0^{i-1},y_0^{i+1},...y_0^n$ respectively, and $\alpha\in \mathbb N^{d},\ \mathbf \gamma^*_i:=(\gamma_1,...,\gamma_{i-1},0,\gamma_{i+1},..., \gamma_{n})\in \mathbb N^{n}$ are multi-indices satisfying $|\alpha|_1 \leq t_d$, $|\gamma_i^*|_\infty\leq 1$, then the mixed derivative $D^{\alpha}_{x}D^{\gamma_i^*}_{\mathbf y}q(x_0, y_0^1,...,y_0^{i-1}, \cdot, y_0^{i+1},...,y_0^{n})\in H^1(\Gamma_i)\hookrightarrow L^\infty(\Gamma_i)$. Therefore, by repeated application of the 1-dimensional Sobolev Imbedding Theorem \cite{Adams2003}  
\begin{align*}
\|q\|_{L^\infty(D\times \Gamma)}&= \max_{x\in D, y\in \Gamma^n}|q(x,y)| = \max_{y\in \Gamma^n} \|q(\cdot, y)\|_{L^\infty(D)}\leq C \max_{(y_1,... ,y_{n})\in\Gamma^n} \|D^\alpha_x q(\cdot,y)\|_{H^1(D)}\\
&\leq C \max_{(y_1,... ,y_{n-1})\in\Gamma^{n-1}} \left(\sum_{|\alpha|_1 \leq d_s}  \int_D (\max_{y_n\in\Gamma_n}|D_{x}^{\alpha} q(x, y_1,... ,y_{n})|)^2 dx\right)^{\frac{1}{2}}\\
& \leq C C_{\Gamma_n} \max_{(y_1,... ,y_{n-1})\in\Gamma^{n-1}} \left(\sum_{|\alpha|_1 \leq d_s}\sum_{\gamma_n=0}^1  \int_D \int_{\Gamma_n}|D_{x}^{\alpha}D^{\gamma_n}_{y_n} q(x, y_1,... ,y_{n})|^2 d\omega\;dx\right)^{\frac{1}{2}}\\
&\leq ...\\
&\leq C \prod_{i=1}^nC_{\Gamma_i} \left(\sum_{|\alpha|_1 \leq t_d} \sum_{|\gamma|_{\infty}\leq 1}\int_D\int_{\Gamma^n} |D_{x}^{\alpha} D_{y}^{\gamma}q(x,y)|^2\rho_n(y)dy\;dx\right)^{\frac{1}{2}}  = \widetilde C_n \|q\|_{\hmix}
\end{align*}
for some constant $\widetilde C_n>0$, independent of $q$, but possibly dependent on the total dimension $d = d_p +n$.\\
\end{proof}

\subsection{Differentiability and Existence of Lagrange Multipliers}

The Fr\'echet differentiability of the equality constraint $e_n(q,u)$ follows directly from its continuity in $q$ and $u$, since $e_n(q,u)$ is affine linear in both arguments. Continuity in $u$ is straightforward. For $u, \tilde u\in\honeofin(D)$,
\begin{align*}
\|e_n(q,u-\tilde u)\|_{\honeofin}^2 = \int_{\Gamma^n}\int_D q|\nabla(u-\tilde u)|^2 dx\;\rho_n\;dy\leq \qmax \|u-\tilde u\|_{\honeofin}^2. 
\end{align*}
Continuity in the parameter $q$ can now also be established, thanks to Lemma \ref{lemma: hmix imbeds continuously in Linfty}. Indeed,
\begin{align*}
\|e_n(q-\tilde q,u)\|_{\honeofin}^2 = \int_{\Gamma^n}\int_D |(q-\tilde q) \nabla u|^2 \; dx\; \rho_n\; dy 
\leq \|q\|_{L^\infty(D\times \Gamma)} \|u\|_{\honeofin}^2\leq \widetilde C_n^2 \|q\|_{\hmix}^2 \|u\|_{\honeofin}^2
\end{align*}
for any $q, \tilde q \in \hmix$. A simple calculation then reveals that the first derivative of $e_n$ in the direction $(h,v)\in \hmix\times\honeofin$ is given by:
\begin{equation}\label{eqn:derivative_of_equality_constraint}
D_{(q,u)}[e_n(q,u)](h,v) = D_q[e_n(q,u)]h + D_u[e_n(q,u)]v \in \ \honeofin,
\end{equation}
where the partial derivatives satisfy
\begin{align*}
\langle D_q[e_n(q,u)]h, \phi \rangle_{\honeofin} &= \int_{\Gamma^n}\int_D h \nabla u \cdot \nabla\phi\; dx\;\rho_n dy = \langle h\nabla u, \nabla \phi \rangle \hspace{0.5 cm}\hbox{and}\\
\langle D_u[e_n(q,u)]v,\phi\rangle_{\honeofin} &= \int_{\Gamma^n} \int_D q\nabla v \cdot \nabla \phi\;dx \;\rho_n dy = \langle q\nabla v,\nabla \phi\rangle\hspace{0.5cm} \hbox{for all } \phi\in \honeofin.
\end{align*}
\noindent We can now derive more traditional, gradient-based first order necessary optimality conditions.

\begin{thm}[Existence of Lagrange Multipliers]\label{thm: existence finite noise lagrange multiplier}
Let $(q^*,u^*)$ be a minimizer for problem \eqref{eqn: finite noise problem}. Then there exists a unique Lagrange multiplier $\lambda^*\in \honeofin$  for which the Lagrange functional $L:\hmix\times\honeofin\times \honeofin\rightarrow \R$, defined by
\begin{equation*}
L(q,u;\lambda):= J(q,u)+\langle \lambda, e_n(q,u)\rangle_{\honeofin}
\end{equation*}
satisfies
\begin{equation}\label{eqn: laplace stationarity}
D_{(q,u)}[L(q^*,u^*;\lambda^*)](h,v)\geq 0 \hspace{2pt} \hbox{ for all $(h,v)\in C(q^*)\times \honeofin$},
\end{equation}
where
\begin{equation*}
C(q^*) = \{ l (c-q^*) : c\in \Qad, \ 0\leq l\in \R\}.
\end{equation*}
Particularly, the adjoint equation and complementary condition hold
\begin{align}
& \langle q^*\nabla \lambda^*, \nabla \phi \rangle = -\langle u^*-\hat u^n,\phi\rangle_{\honeofin} \label{eqn: adjoint equation}\\
&\beta \langle q^*,q-q^*\rangle_{\hmix}+\langle (q-q^*)\nabla u^*, \nabla \lambda^* \rangle\geq 0 \ \ \ \hbox{for all $q\in \Qad$}. \label{eqn: complementary condition}
\end{align}
\end{thm}
\begin{proof}
Let $(q^*,u^*)$ be a minimizer of problem \eqref{eqn: finite noise problem}. We  show that $(q^*,u^*)$ satisfies the regular point condition 
\begin{equation}\label{thm_lag_existence_eq1}
D_{(q,u)}[e_n(q^*,u^*)](C(q^*)\times \honeofin)=\honeofin,
\end{equation}
from which the existence of the Lagrange multiplier follows directly by \cite{Maurer1979}. In light of \eqref{eqn:derivative_of_equality_constraint}, this amounts to establishing the existence of solutions $(h,v)\in C(q^*)\times\honeofin$ to the equation 
\[
D_q[e_n(q^*,u^*)]h + D_u[e_n(q^*,u^*)]v = w,
\]
for arbitrary $w\in \honeofin$. Since $0\in C(q^*)$ and the finite noise elliptic equation 
\[
\int_{\Gamma^n}\int_D q \nabla v \cdot \nabla \phi\; dx\;\rho_n dy = \int_{\Gamma^n}\int_D \nabla w\cdot \nabla \phi\; dx\;\rho_n dy \ \ \forall \phi\in \honeofin
\] 
is solvable for any $w\in \honeofin$, condition \eqref{thm_lag_existence_eq1} is satisfied and hence there exists a Lagrange multiplier $\lambda^*\in \honeofin$ such that \eqref{eqn: laplace stationarity} holds. More explicitly, 
\begin{align}\label{thm_lag_existence_eq2}
0&\leq \langle u^*-\hat u, v \rangle_{\honeofin} +\beta\langle q^*,h\rangle_{\hmix}  + \langle D_q[e_n(q^*,u^*)]h + D_u[e_n(q^*,u^*)]v,\lambda^*\rangle_{\honeofin}  \nonumber \\
& = \langle u^*-\hat u, v \rangle_{\honeofin}+\beta \langle q^*, h\rangle_{\hmix} + \langle h \nabla u^*,\nabla \lambda^*\rangle + \langle q^*\nabla v,\nabla\lambda^*\rangle
\end{align}
for all $(h,v)\in C(q^*)\times \honeofin$. In particular, if $h=0$, we obtain 
\begin{equation*}
\langle q^* \nabla\lambda^*, \nabla v \rangle
 =  - \langle u^*-\hat u, v \rangle_{\honeofin}\ \ \ \ \ \ \hbox{for all $v\in \honeofin$}, 
\end{equation*}
which yields the adjoint equation \eqref{eqn: adjoint equation}.
The uniqueness of $\lambda^*$ now follows directly from the uniqueness of the solution to the elliptic equation \eqref{eqn: adjoint equation}. 
Finally, setting $v=0$ and $h = q - q^*$ in (\ref{thm_lag_existence_eq2}) for any $q\in \Qad$ yields the complementary condition \eqref{eqn: complementary condition}
\begin{equation*}
\beta \langle q^*,q-q^*\rangle_{\hmix}+\langle (q-q^*)\nabla u^*,\nabla\lambda^*\rangle \geq 0 \ \ \ \ \ \hbox{for all $q\in \Qad$}.
\end{equation*}
\end{proof}

\section{An Augmented Lagrangian Algorithm}\label{section:augmented_lagrangian_alg}

With the availability of derivative information, the finite noise problem \eqref{eqn: finite noise problem} can now be solved by more conventional optimization algorithms. We make use of the augmented Lagrangian method, an iterative approach that may be viewed as a modified penalty method. The quadratic penalty method avoids explicit enforcement of the equality constraint $e_n(q,u)=0$ by incorporating an additional term, that penalizes violations of the constraint, into the cost functional. For example in \eqref{eqn: finite noise problem}, this could require solving a series of sub-problems of the form
\begin{equation}\label{eqn: quadratic penalty method}
\min_{(q,u)\in \Qad \times \honeofin} \frac{1}{2}\|u-\uhat\|_{\honeofin}^2 + \frac{\beta}{2}\|q\|_{\hmix}^2 + \frac{c_k}{2}\|e_n(q,u)\|_{\honeofin}^2,
\end{equation}
where the sequence $\{c_k\}_{k=0}^\infty$ increases steadily as $k\rightarrow \infty$. In fact, the convergence of this class of methods requires $\lim_{k\rightarrow\infty} c_k = \infty$, leading to a progressive deterioration in the conditioning of the sub-problem. 

\vspace{1em}

The augmented Lagrangian method avoids this conditioning issue by instead solving the sequence of problems 
\begin{equation}\label{eqn: auxiliary problem}\tag{$P_\mathrm{aux}$}
\min_{(q,u)\in \Qad\times \honeofin} L_{c_k}(q,u,\lambda^k), 
\end{equation}
where $\{c_k\}_{k=0}^\infty$ is a non-decreasing sequence of positive numbers and the augmented Lagrangian functional, $L_{c_k}:\hmix\times\honeofin\times\honeofin\rightarrow \R$, is given by
\[
L_{c_k}(q,u,\lambda^k) =  \frac{1}{2}\|u-\uhat^n\|_{\honeofin}^2 + \frac{\beta}{2}\|q\|_{\hmix}^2 + \langle \lambda^k, e_n(q,u)\rangle_{\honeofin} + \frac{c_k}{2}\|e_n(q,u)\|_{\honeofin}^2. 
\]
The function $\lambda^k\in \honeofin$ is an approximation of the Lagrange multiplier defined in \eqref{eqn: adjoint equation} and is updated via $\lambda^{k+1} = \lambda^{k} + c_k e_n(q^k,u^k)$, where $(q^k,u^k)$ minimizes \eqref{eqn: auxiliary problem}. More explicitly,

\begin{algorithm}[H]\caption{The Augmented Lagrangian Algorithm}\label{alg: augmented lagrangian method}
\SetKwInOut{Input}{Input}
\SetKwInOut{Output}{Output}
\Input{$\hat u$}
\Output{$q$}
Choose $\lambda^0 \in H_0^1(D)$, and non-decreasing sequence $\{c_k\}$ with $c_0 >0$\;
Set $k=0$\;
\While{not converged}{
    Obtain minimizers $(q^k, u^k)$ by solving the auxiliary problem \eqref{eqn: auxiliary problem}\;
    Set $\lambda^{k+1}:= \lambda^{k} + c_k e_n(q^k,u^k)$\;
    Set $k = k + 1$ and test for convergence\;
}
\end{algorithm}

This algorithm, developed in \cite{Hestenes1969,Powell1978}, has been used extensively for deterministic parameter identification- and control problems in elliptic systems \cite{Ito1991,Ito1990,Kunisch1997}. Unlike for penalty methods, the sequence $\{c_k\}_{k=0}^\infty$ is not required to grow without bound to guarantee convergence.

\vspace{1em}

It was shown in \cite{Ito1990} and \cite{Kunisch1997} (Theorems 2.4, 2.5, and subsequent remarks) that the iterates $(q^k,u^k,\lambda^k)$ computed by Algorithm \ref{alg: augmented lagrangian method} converge to the minimizers $(q^*,u^*,\lambda^*)$ of \eqref{eqn: finite noise problem}, under the following second-order sufficient optimality condition:
\begin{assumption}\label{ass: second order sufficient optimality}
Assume there exists a constant $\tau = \tau(\beta)>0$ so that
\begin{equation*}
D^2_{(q,u)}[L(q^*,u^*, \lambda^*)] (h,v)^2 \geq \tau (\|h\|_{\hmix}^2 + \|v\|_{\honeofin}^2) \ \ \ \hbox{for all } (h,v)\in \hmix\times\honeofin.
\end{equation*}
\end{assumption}
The original convergence proof, formulated in a general Hilbert space setting, carries over directly to our problem. We refer the interested reader to the cited references. 
 Moreover, the cost functional $L_{c_k}$ appearing in the auxiliary problem \eqref{eqn: auxiliary problem} is quadratic in $q$ for fixed $u$ and $\lambda$ and quadratic in $u$ for fixed $q$ and $\lambda$, suggesting the use of sequential splitting methods to speed up the solution of the auxiliary subproblem. To wit, the subproblem \eqref{eqn: auxiliary problem} in Algorithm \ref{alg: augmented lagrangian method} is replaced with the sequence: Solve\\
\begin{equation}
\min_{q\in Q_\mathrm{ad}} L_{c_k}(q, u_{n,k}^*, \lambda_{n,k}^*). \label{eqn:aug_lag_aux_q}\tag{$P_{\mathrm{aux}}^q$}
\end{equation}
for $q_{n,k}^*$, then obtain $u_{n,k+1}^*$ by solving the minimization problem
\begin{equation}
\min_{u\in H_0^1} L_{c_k}(q_{n,k+1}^*,u, \lambda_{n,k}^*).\label{eqn:aug_lag_aux_u}\tag{$P_{\mathrm{aux}}^u$}
\end{equation}
 
\setlength{\algomargin}{2em}    
\begin{algorithm}[H]\label{alg: augmented lagrangian method with splitting}
\SetKwInOut{Input}{Input}
\SetKwInOut{Output}{Output}
Choose $\lambda_{n,0}\in h(D)$, and non-decreasing sequence $\{c_k\}$ with $c_0 >0$\;
Set $k=0$ \;
\While{not converged}{
    Solve the auxiliary problem sequentially, i.e. for iterates $q_{n,k+1}^*$ and $u_{n,k+1}^*$\;
    \hspace{1em} Get $q_{n,k+1}^*$ by solving problem \eqref{eqn:aug_lag_aux_q} (using current values of $u_{n,k}^*$ and $ \lambda_{n,k}^*$)\;
	\hspace{1em} Get $u_{n,k+1}^*$ by solving problem \eqref{eqn:aug_lag_aux_u} (using current values of $q_{n,k+1}^*$ and $\lambda_{n,k}^*$)\;
    Set $\lambda_{n,k+1}^*:= \lambda_{n,k}^* + c_k e_n(q_{n,k+1}^*,u_{n,k+1}^*)$\;
    Set $k = k + 1$ and test for convergence.
}
\caption{The Augmented Lagrangian Algorithm with Sequential Splitting}
\end{algorithm}

We consider the auxiliary sub-problems \ref{eqn:aug_lag_aux_q} and \ref{eqn:aug_lag_aux_u} in more detail. The unconstrained minimizer $u_{n,k+1}^*$ of \ref{eqn:aug_lag_aux_u} can be computed simply by solving the first order optimality system $D_u \big[ L_{c_k}(q,u,\lambda)\big](v) = 0$ for all $v\in \honeofin$ and fixed $q\in \Qad, \lambda \in \honeofin$, where 
\begin{align}
0 =\ & D_u \big[ L_{c_k}(q,u,\lambda)\big](v) \nonumber\\
 = \  & \langle u - \uhat, v\rangle_{\honeofin} + \langle \lambda, D_u[e_n(q,u)](v) \rangle_{\honeofin} + c_k \langle e_n(q,u), D_u[e_n(q,u)](v)\rangle_{\honeofin}\nonumber\\
 = \ & \langle u - \uhat, v\rangle_{\honeofin} + \langle q \nabla \lambda, \nabla v\rangle + c_k \langle q \nabla e_n(q,u),\nabla v \rangle \nonumber\\ 
= \ & \langle \nabla u + c_k q \nabla e_n(q,u), \nabla v \rangle - \langle \uhat - q\nabla \lambda, \nabla v \rangle.\label{eqn:aug_lag_grad_u}
\end{align}

The first order optimality system for \ref{eqn:aug_lag_aux_q} if $q\in \mathrm{int}(\Qad)$ amounts to setting $D_q[ L_{c_k}(q,u,\lambda)](h) = 0$ for all $h\in \hmix$. More specifically, 
\begin{align}
0 =\ & D_q[ L_{c_k}(q,u,\lambda)](h) \nonumber \\
= \ & \beta \langle q,h\rangle_{\hmix} + \langle \lambda, D_q[e_n(q,u)](h) \rangle_{\honeofin} + c_k \langle e_n(q,u), D_q[e_n(q,u)](h)\rangle_{\honeofin} \nonumber \\
= \ & \beta \langle q,h\rangle_{\hmix} + \langle h \nabla \lambda,\nabla u  \rangle + c_k \langle h \nabla e_n(q,u), \nabla u \rangle .\label{eqn:aug_lag_grad_q}
\end{align}

\section{Numerical Discretization}\label{section: discretization}

This section details the numerical discretization of the augmented Lagrangian method (Algorithm \ref{alg: augmented lagrangian method with splitting}) outlined in the previous section.  We approximate the parameter- $q$,  state- $u$, and adjoint random fields $\lambda$ spatially by means of piecewise polynomial basis functions related to finite element meshes of the spatial domain $D$. For the deterministic parameter identification problem, it was observed in \cite{Ito1991} that using a coarser mesh for the parameter space than for the state space amounts to an implicit regularization. For our numerical experiments, we therefore base our approximation of $q$ on a coarser triangulation $\mathcal T_q$ of $D$ with associated finite element space $V_q = \mathrm{span}\{\phi_1^q,...,\phi_{M_q}^q\}$, while estimating $u$ and $\lambda$ based on the finer grid $\mathcal T_u$, in our case a uniform refinement of $\mathcal T_q$, with associated subspace $V_{u} = \mathrm{span}\{\phi_{1}^u,...,\phi_{M_u}^u\}$. The spatial approximation $v^{M_u}\in V_u\otimes L^2(\Omega)$ of $v\in \honeoinf$ can be written explicitly as 
\[
v^{M_u}(x,\omega):=\sum_{i=1}^{M_u} v(x_i,\omega)\phi_i^u(x).
\]
Estimates of associated spatial inner products can be also be computed using the mass- and stiffness matrices defined component-wise by 
\[
A^u: = \left[\int_D \phi_{i_1}^u(x) \phi_{i_2}^u(x)\;dx \right]_{i_1,i_2=1}^{M_u} \ \ \ \hbox{and } \ A_x^u := \left[ \int_D \nabla\phi_{i_1}^u(x)\cdot \nabla\phi_{i_2}^u(x) \;dx \right]_{i_1,i_2=1}^{M_u}
\]
respectively. Similar expressions hold for the spatial approximations $h^{M_q}\in V_q\otimes L^2(\Omega)$ of random fields $h\in \htilde$ and for the mass- and stiffness matrices $A^q$ and $A_x^q$ on $V_q$, although we assume here that homogeneous Dirichlet boundary conditions are incorporated into the construction of  $A_x^u$, rendering it invertible, while no such conditions are imposed on $A_x^q$.

\subsection{Karhunen-Lo\`eve Expansion of the Data}

In order to reduce our variational problem \eqref{eqn: infinite L2 problem} to its \lq finite noise' approximation \eqref{eqn: finite noise problem}, we must first approximate the truncated KL expansion of the measured data $\uhat\in \honeoinf$, which in turn requires the spectral decomposition of the compact covariance operator $\mathscr C_{\uhat}:H^1_0(D) \rightarrow H^1_0(D)$, defined in terms of its covariance kernel 
\begin{align*}
C_{\uhat}(x,x') &= \Exp[(\uhat(x')-u_0(x'))(\uhat(x)-u_0(x))] \\
v\in H^1_0(D) &\mapsto\left(\mathscr C_{\uhat} v\right)(x') = \int_D \nabla_x C_{\uhat}(x,x')\cdot \nabla v(x)\; dx \in H^1_0(D),
\end{align*}
where $u_0(x) := \Exp[\uhat(x,\cdot)]$.  In practice, $\uhat$ commonly occurs in the form of an data matrix $\hat{\mathbf U} = [\uhat_{i,j}]$, where $\uhat_{i,j} = \uhat(x_i,\omega_j)$ denotes the $j^{th}$ random sample of the field obtained at spatial location $x_i$ for $j = 1,...,N_\mathrm{sample}$. We assume here that this data is either sampled at the vertices $x_i$ of the grid $\mathcal T_u$, or that it is interpolated, using splines for example, so that $\hat{\mathbf U}$ is of size $M_u$ by $N_\mathrm{sample}$. Let the sample mean $\mathbf m = [m_1,...,m_{M_u}]^T$ and covariance matrix $\Sigma = [\sigma_{i_1,i_2}]_{i_1,i_2=1}^{M_u}$ be defined componentwise by 
\begin{align*}
m_i & := \frac{1}{N_\mathrm{sample}} \sum_{j=1}^{N_\mathrm{sample}} \uhat(x_i,\omega_j), \ \text{and } \\
\sigma_{i_1,i_2} &:= \frac{1}{N_{\mathrm{sample}}}\sum_{j=1}^{N_\mathrm{sample}} \left(\uhat(x_{i_1},\omega_j)- m_{i_1})(\uhat(x_{i_2},\omega_j)-m_{i_2}\right),
\end{align*}
respectively. The sample mean $\uhat_0^{M_u}$ and covariance $C_{\uhat}^{M_u}$ of a finite element representation $\uhat^{M_u}$ of $\uhat$ then take the form 
\begin{align*}
\uhat_0^{M_u}(x) = \sum_{i=1}^{M_u} m_i \phi_i^u(x) \qquad \text{and } \\
C_{\uhat}^{M_u}(x,x') = \sum_{i_1,i_2} \sigma_{i_1,i_2}\phi_{i_1}^u(x)\phi_{i_2}^u(x'),
\end{align*}
respectively. This allows us to form the finite element approximation $\mathscr C_{\uhat}^{M_u}:V_u \rightarrow V_u$ of the covariance operator by letting
\begin{align*}
\left(\mathscr C_{\uhat}^{M_u} v\right)(x') &= \int_D C_{\uhat}^{M_u}(x,x')v(x)\;dx\\
 &= \sum_{i_1=1}^{M_u}v(x_{i_1})\phi_{i_1}(x')\left(\sum_{i_2=1}^{M_u} \sigma_{i_1,i_2}\int_D \nabla \phi_{i_1}(x)\cdot \nabla \phi_{i_2}(x)\;dx \right)
\end{align*}
for any element $v\in V_u$. The operation $\mathscr C_{\uhat}^{M_u}v$ can also be expressed in terms of the spatial coordinatization $\mathbf v = [v(x_1),...,v(x_{M_u})]^T$ of $v$ as the matrix-vector product $\Sigma A_x^u \mathbf v$ and hence the spectral decomposition of $\mathscr C_{\uhat}^{M_u}$ amounts to finding the eigenpairs $(\nu,\mathbf b)$ so that $\Sigma A_x^u \mathbf b = \nu \mathbf b $, or equivalently the generalized eigenvalue problem $A_x^{u}\Sigma A_x^{u} \mathbf{b} = \nu A_x^u\mathbf b$. By virtue of the positive semi-definiteness of the discretized covariance operator $\mathscr C_{\uhat}^{M_u}$ the eigenvectors $\mathbf b$ are orthogonal, so that the associated eigen-decomposition takes the form $\Sigma A_x^u = B D^\nu B^T$ with $D^\nu$ diagonal and $B$ unitary. The truncated KL expansion amounts to a projection of the data onto the eigenspace associated with the largest $n$ eigenvalues. The compactness and semi-positive definiteness of the operator $\mathscr C_{\uhat}$ ensure that its spectrum is countable with an accumulation point at $0$, allowing us to determine a suitable truncation level $n$ by estimating the rate of decay of the eigenvalues. Since $\mathscr C_{\uhat}^{M_u}$ only has finite rank, however, this criterion is subject to the level of spatial discretization $M_u$, i.e. we require $n\leq M_u$. The truncated, discretized KL expansion $\uhat^{n,M_u}$ of the field $\uhat$ now takes the form 
\begin{equation*}
\uhat^{n,M_u}(x,\omega)= \uhat_0^{M_u}(x) + \sum_{k=1}^{n} \sqrt{\nu_k}b_k^{M_u}(x)Y_k(\omega) \ \text{ for } \omega \in \Omega, 
\end{equation*}
where $Y(\omega) = [Y_1(\omega),...,Y_n(\omega)]^T$ is a random vector whose joint density function can be estimated from samples obtained by projecting the centered data matrix onto the subspace spanned by the dominant $n$ eigenvectors. Indeed, let $B_n$ be the matrix consisting of the first $n$ columns of $B$ and $D^\nu_n = \mathrm{diag}(\nu_1,...,\nu_n)$. Then
\begin{align*}\label{eqn: define random basis}
Y_k(\omega_j)  &= \frac{1}{\sqrt{\nu_k}} \int_D \nabla \left(\uhat^{n,M_u}(x,\omega_j) - \uhat_0^{M_u}(x)\right) \cdot \nabla b_k^{M_u}(x)\;dx \\
& = \sum_{i_1,i_2=1}^{M_u} \frac{1}{\sqrt{\nu_k}} \left(\uhat(x_{i_1},\omega_j - \uhat_0(x_{i_1})\right) b_k(x_{i_2})\int_D \nabla \phi_{i_1}^u(x)\cdot \nabla \phi_{i_2}^u(x)\;dx 
\end{align*}
for $k=1,...,n$, so that $Y(\omega_j) = (D_n^\nu)^{-\frac{1}{2}}B_n^T A_x^u \left( \mathbf{\hat U}(:,j) - \mathbf m\right)$ for $j=1,...,N_{\mathrm{sample}}$. It is from these samples that the joint density function $\rho_n$ can be estimated. The KL expansion discussed in this paper differs slightly from the usual approach \cite{Schwab2006}, in that we are defining the covariance operator on the Hilbert space $H^1_0(D)$ instead of on $L^2(D)$, to ensure convergence of the projection in the $\honeoinf$ norm. In practice, this choice of the norm doesn't make a significant difference in computations. 

\vspace{1em}

The estimation of multidimensional density functions is a highly non-trivial problem in general and an active field of current statistical research, well beyond the scope of this paper. The reader is referred to the books \cite{Scott2005,Klemelae2009}, as well as the survey article \cite{Scott1992}, for a more exhaustive treatment of the subject. 
The random vectors encountered in Section \ref{section: numerical examples} are only of moderate size and we either assume to know their joint densities or make use of kernel density estimators to approximate them empirically.

\subsection{Discretization in the Stochastic Component}

The choice of the type of nodal basis used to discretize the state equation \eqref{eqn: finite noise problem} or the adjoint system \eqref{eqn: adjoint equation} depends on the smoothness of the fields $u$ and $\lambda$ as functions of $y$. Under certain smoothness conditions on the parameter $q(x,y)$, which are readily satisfied if $q$ is written in terms of its KL expansion, the model output $u(x,y)$ can be shown to be analytic in $y$, warranting the use of global interpolating basis functions such as Lagrange polynomials \cite{Babuska2007}. In our case $q(x,y)$ is written in terms of the random variables in the KL expansion of the measured data $\uhat$ and hence such smoothness conditions may no longer hold. Consequently, neither the model output $u$, nor the Lagrange multiplier $\lambda$, characterized by the adjoint equation, are guaranteed to exhibit the requisite smoothness as functions of $y$ to allow for their approximation by a global polynomial basis. Here we make use of an interpolating basis of piecewise smooth, multi-linear hat functions. \\

Assume, without loss, of generality that the stochastic domain $\Gamma^n = [0,1]^n$. While much is known about interpolation formulas on one-dimensional domains, the problem of computing efficient and accurate multi-dimensional interpolants remains a challenge. Sparse grid methods \cite{Barthelmann2000, Gerstner1998, Novak1996, Smolyak1963} efficiently combine one-dimensional interpolation schemes to obtain accurate interpolants in higher dimensions with only a moderate number of grid points. Suppose $\Gamma^n$ is subdivided along each dimension into one-dimensional grids $X^{l_t}$, $t =1,2,...,n$ of equally spaced points, where the multi-index $\mathbf l = (l_1,...,l_{n})\in \mathbb N^{n}$ denotes the level of refinement in each direction. In particular, each grid $X^{l_t}$  consists of nodes $\{y_{l_t,j_t}\}_{j_t=0}^{m^{l_t}}$, where 
\[m^{l_t} = \left\{\begin{array}{ll} 1, & \hbox{if $l_t=1$}\\ 2^{l_t}, & \hbox{if $l_t > 1$}\end{array}\right.\ \ \hbox{and } \ y_{l_t,j_t}=\left\{\begin{array}{ll} 0.5, & \text{if } l_t = 1, j_t=1\\ 2^{-l_t}j_t, &\text{if $l_t>1$, for } j_t=0,1,...,m^{l_t}\end{array}\right. . \] 
For convenience, we define $ m^{l}:=(m^{l_1},...,m^{l_n})$ and take $j\leq  m^l$ to mean $j_t\leq m^{l_t}$ for each $t=1,..,n$. The full tensor product grid $X^{l}$ on $\Gamma^n$, given by
\[X^{l} := X^{l_1}\times\cdots \times X^{l_n},\]
thus consists of the points $\{y_{l,j}\}_{j\leq m^{l}}$.
Let $\{\psi_{l_t,j_t}\}_{j_t=0}^{m^{l_t}}$ denote a set of one-dimensional, nodal interpolating basis functions centered at the grid points $\{y_{l_t,j_t}\}_{j_t=0}^{m^{l_t}}$ of each one-dimensional grid $X^{l_t}$, $t=1,...,n$. 
We use bases of one-dimensional piecewise linear hat functions, defined for any point $y\in [0,1]$ by $\psi_{l_t,j_t}(y) :=1 $ when $l_t = 1$ and
\[
\psi_{l_t, j_t}(y): = 
\psi\left(m^{l_t}\left(y-\frac{j_t}{m^{l_t}} \right)\right), \ \ \ \ \psi(z) :=\left\{\begin{array}{cc} 1-|z|, & \ \ \hbox{if } -1\leq z\leq 1 \\ 0, & \hbox{otherwise}\end{array}\right. ,\]
when $l_t> 1$. A basis function $\psi_{l, j}$ centered at a node $y_{l,j} = (y_{l_1,j_1},...,y_{l_n, j_n})$ in the multi-dimensional grid $X^l = X^{l_1}\times ...\times X^{l_n}\subset [0,1]^n$ can then be obtained by taking the product of the appropriate univariate nodal basis functions, i.e. for any $y = (y_1,...,y_n)\in [0,1]^n$,
\[
\psi_{l,j}(y)= \psi_{l_1,j_1}\otimes\cdots\otimes\psi_{l_n,j_n}(y):= \prod_{t=1}^{n} \psi_{l_t, j_t}(y_t).
\]
Note that the one-dimensional grids are nested, i.e. $X^0\subset X^1\subset ...\subset X^{l_t}$ for any $l_t\in \Natural$. As a result, the subspaces spanned by one-dimensional interpolating basis functions are also nested and hence it is relatively straightforward to compare the accuracy of one-dimensional grids with various refinement levels $l_t$. A multi-dimensional interpolation formula with refinement level $L$ in each direction can be obtained by combining the one-dimensional interpolation formulas 
\[
U^{L}(v) = \sum_{j_t=0}^{m^L} v(y_{l_t,j_t})\psi_{l_t,j_t}
\]
to form the full tensor multi-variate interpolant
\[
\left(U^{L}\otimes\cdots\otimes U^{L}\right)(v) = \sum_{j\leq m^L} v(y_{l,j})\psi_{l,j}.
\]
The number of grid points needed to construct this interpolant is $(m^L)^n$, which scales exponentially as the dimension $n$ of the space increases. 

\vspace{1em}

The sparse grid interpolant $A^L(v)$ with interpolation level $L\geq 0$ is constructed from linear combinations of lower order full tensor interpolants as follows 
\begin{equation}\label{eqn: sparse grid definition}
A^L(v) = \sum_{1\leq |l|_1\leq L+n-1} (-1)^{N-|l|_1} \left(\begin{array}{c} n-1 \\ L -|l|_1 \end{array}\right) \left(U^{l_1}\otimes\cdots\otimes U^{l_n}\right)(v).
\end{equation}

Through cancellation, the effective number of grid points required is much lower than that of the full tensor product, while its accuracy is only marginally worse. 

\vspace{1em}

In practice, formula \eqref{eqn: sparse grid definition} is not used directly to construct interpolants. Instead, higher order interpolants are constructed recursively from lower order ones by adding corrections on the appropriately refined grid. This is achieved through the use of hierarchical basis functions, defined for every level $l = (l_1,...,l_n)$ to be the span $W^l(\Gamma^n) = \mathrm{span}\{\psi_{l,j}:j\in J_l\}$, where 
\[
J_l=\left\{j\in \Natural^n: j_t = \left\{ \begin{array}{ll}
1/2 & \text{if } l_t = 1,\\
0 \text{ or } 1 & \text{if } l_t = 2, \\
\text{an odd number in }\{1,...,m^{l_t}-1\} & \text{if } l_t \geq 3  
\end{array}\right. \right\}.
\]
Indeed, it can be shown (see \cite{Bungartz2003}) that $A^1(v) = ( U^1\otimes \cdots \otimes U^1) (v)$, while for any $L > 1$
\begin{equation*}
A^L(v) = A^{L-1}(v) + \Delta A^L(v),
\end{equation*}
where
\[
\Delta A^{L}(v)= \sum_{|l|_1 = L+n-1}\;\sum_{j \in J_l} \left[v(y_{l,j}) - A^{L-1}(v)(y_{l,j})\right]\cdot \psi_{l,j}(y).\]
The coefficients $v_z(y_{l,j}) = v(y_{l,j}) - A^{L-1}(v)(y_{l,j})$ appearing in the update $\Delta A^L$, also known as hierarchical surpluses, represent the discrepancy between the function $v$ and the $L-1$ level interpolant $A^{L-1}(v)$ at the new gridpoints. Hierarchical surpluses provide useful {\em a posteriori} error estimates that can readily be employed by an adaptive scheme to identify the regions where the grid should be refined \cite{Bungartz2003,Ma2009,Ma2010}. Unfortunately, it is difficult to incorporate adaptive approximation seamlessly into these high-dimensional gradient-based optimization methods. Since the functions $q_k, u_k$ and $\lambda_k$ are changing at each iteration of the optimization algorithm, the adaptive refinement scheme would have to be adjusted throughout the duration of the algorithm. This can be costly, especially in light of the fact that the relevant bilinear- and trilinear forms would have to be updated after each adaptive refinement or coarsening. 

\vspace{1em}

For the sake of notational expediency, we let $j=1,...,N$ be an enumeration of the sparse grid points, i.e. 
\[
\{y_j\}_{j=1}^N = \{y_{l,j}: 1\leq |l|_1\leq L+n-1, \ j\in J_l\},
\]
so that the stochastic sparse grid interpolant $v^{N}(x,y)$ of $v\in \honeofin(D)$ takes the form
\[
v^N(x,y) = \sum_{j=1}^N v_z(x,y_j)\psi_j(y),
\]
while the full approximation of $v$ is given by
\[
v^{M_u,N}(x,y) = \sum_{i=1}^{M_u}\sum_{j=1}^N v_z(x_i,y_j) \phi_i^u(x) \psi_j(y)
\]
The function values $v(x_i,y_j)$ can be related to the hierarchical surpluses $v_z(x_i,y_j)$ by means of a linear, invertible transformation. 


\subsection{The Discretized Optimization Problem}

To approximate the inner products and bilinear forms appearing in optimization Algorithm \ref{alg: augmented lagrangian method with splitting}, we require the deterministic bilinear forms introduced earlier, the $\rho$-weighted stochastic bilinear forms $S_\rho$ and $S_\rho^{\mathrm{mix}}$, and the stochastic trilinear form $T_\rho$, defined componentwise as follows
\begin{align*}
S_\rho& = \left[\int_{\Gamma^{n}}  \psi_{i_1}(y) \psi_{i_2}(y)\;\rho_{n}(y)\;dy\right]_{i_1,i_2=1}^{N},\\
S_\rho^{\mathrm{mix}} &= \left[\sum_{|\gamma|_\infty\leq s}\int_{\Gamma^{n}} D_y^\gamma \psi_{i_1}(y)D_y^\gamma \psi_{i_2}(y)\;\rho_{n}(y)\;dy,\right]_{i_1,i_2=1}^{N},\ \hbox{ and} \\
T_\rho&= \left[\int_{\Gamma^{n}} D^\gamma_y \psi_{i_1}(y)D^\gamma_y \psi_{i_2}(y)\psi_{i_3}(y)\;\rho_{n}(y)\;dy\right]_{i_1,i_2,i_3=1}^{N}.
\end{align*}
The evaluation of these multi-dimensional integrals for any given density function $\rho$ is a challenging task in general, although they can be computed offline. Note that, whereas each basis function $\psi_j(y)$ can be written as the product of appropriate one-dimensional basis functions, the $\rho$ cannot in general be decomposed as the product of its marginals, thus preventing the effective decoupling of these integrals into products of simpler ones.

\vspace{1em}

For any function $v\in \honeofin(D)$, we define $\mathbf{v}^z$  $:= [\mathbf{v}^z_1,...,\mathbf{v}^z_{N}]^T$ to be the vector of hierarchical surpluses where $\mathbf{v}^z_j = [v_z(x_1,y_j),...,v_z(x_{M_u},y_j)]^T$ are the surpluses corresponding to the sparse grid node $y_j$. Let a similar definition hold for functions $h\in \hmix(D)$. The $\honeofin$-inner product of approximations $v^{M_u,N}$ and $w^{M_u,N}$ then take the form
\begin{align*}
&\langle v^{M_u,N}, w^{M_u,N} \rangle_{\honeofin}\\ =&\sum_{i_1,i_2=1}^{M_u}\sum_{j_1,j_2=1}^N v_z(x_{i_1},y_{j_1})w_z(x_{i_2},y_{j_2})\left(\int_{\Gamma^n}\psi_{j_1}\psi_{j_2}\rho\;dy\right)\left(\int_D \nabla \phi_{i_1}^u \cdot \nabla\phi_{i_2}^u\;dx\right)\\
=& \sum_{j_1,j_2=1}^N (\mathbf {v}^z_{j_2})^T A_x^u \mathbf{w}^z_{j_1} = \ (\mathbf{v}^z)^T (S_\rho \otimes A_x^u) \mathbf{w}^z. 
\end{align*}
Similarly, 
\begin{align*}
& \langle v^{M_u,N}, w^{M_u,N} \rangle_{\tilde L^2} =  (\mathbf{v}^z)^T (S_\rho \otimes A^u) \mathbf{w}^z, \ \text{ and }\\ 
& \langle h^{M_q,N}, k^{M_q,N} \rangle_{\hmix} =  (\mathbf h^z)^T (S_\rho^{\mathrm{mix}} \otimes A_x^q) \mathbf{k}^z,
\end{align*}
for any two functions $h,k\in \hmix(D)$. The discretized $q$-weighted bilinear form\\* $\langle q^{M_q,N}\nabla v^{M_u,N},\nabla w^{M_u,N}\rangle$ on the other hand requires the use of the weighted trilinear form $T_\rho$. Indeed
\begin{align*}
&\langle q^{M_q,N} \nabla u^{M_u,N}, \nabla v^{M_u,N}\rangle = \int_{\Gamma^n}\int_D q^{M_q,N} \left(\nabla u^{M_u,N}\cdot \nabla v^{M,N}\right)\;\rho\;dx\;dy\\
=& \sum_{i_1,i_2=1}^{M_u}\sum_{j_1,j_2=1}^{N} u_z(x_{i_1},y_{j_1})v_z(x_{i_2},y_{j_2}) \int_{\Gamma^n}\int_D q^{M_q,N}\nabla\phi^u_{i_1}\cdot \nabla \phi^u_{i_2}\psi_{j_1}\psi_{j_2} \;\rho\;dx\;dy\\
=&\ (\mathbf u^z)^T S_{\rho,q}\mathbf v^z, 
\end{align*}
where $S_{\rho,q}$ is defined componentwise as 
\[
S_{\rho, q} := \left[\sum_{i=1}^{M_q}\sum_{j=1}^N q_z(x_i,y_j) \left(\int_{\Gamma^n} \psi_{j} \psi_{j_1} \psi_{j_2} \;\rho\;dy\right)\left(\int_D \phi^q_i \nabla \phi^u_{i_1}\cdot \nabla \phi^u_{i_2} \;dx\right)\right]_{\substack{i_1,i_2 = 1,...,M_u\\ j_1,j_2 = 1,...,N}}.
\]
Alternatively, 
\[
\langle q^{M_q,N} \nabla u^{M_u,N}, \nabla v^{M_u,N}\rangle = (\mathbf q^z)^T (S_{\rho,u})\mathbf v^z, 
\]
where
\[
S_{\rho,u} := \left[ \sum_{i=1}^{M_u} \sum_{j=1}^N u_z(x_i,y_j)\left(\int_{\Gamma^n} \psi_{j} \psi_{j_1} \psi_{j_2} \;\rho\;dy\right)\left(\int_D \phi^q_{i_1} \nabla \phi^u_{i}\cdot \nabla \phi^u_{i_2} \;dx\right)\right]_{\substack{i_1,i_2 = 1,...,M_u\\ j_1,j_2 = 1,...,N}}.
\]
In our numerical calculations, we approximate the sample paths of the equality constraint $e\in \honeofin(D)$ as solutions to the spatially discretized Poisson problems 
\begin{equation}\label{eqn: equality constraint}
\int_D \nabla e^{M_u}_j\cdot \nabla \phi_{i}^u \; dx = \int_D q^{M_q,N}(\cdot,y_j) \nabla u^{M_u,N}(\cdot,y_j) \cdot \nabla \phi_i^u\; dx - \int_D f  \phi_i^u \; dx,
\end{equation}
$i=1,...,M^u$, or equivalently 
\[
\mathbf e_j = \mathbf e_j(\mathbf q,\mathbf u) - \mathbf e_j(\mathbf f), 
\]
for each $j=1,...,N$, where 
\begin{align*}
\mathbf e_j(\mathbf q,\mathbf u) & = (A_x^u)^{-1} H^j(\mathbf q,\mathbf u),\ \ \ \mathbf e_j(\mathbf{f}) = (A_x^u)^{-1} \mathbf{f}, \qquad\text{and } \\
H^j(\mathbf q,\mathbf u)  &= \left[ \sum_{i_1=1}^{M^q}\sum_{i_2=1}^{M_u}q(x_{i_1},y_j)u(x_{i_2},y_j)\int_D \phi_{i_1}^q \nabla \phi_{i_2}^u\cdot \phi_{i}^u\;dx\right]_{i=1}^{M^u}.
\end{align*}
The vector $\mathbf e = [\mathbf e_1,...,\mathbf e_N]^T$ of sample paths $\mathbf e_j = [e^{M_u}(x_1,y_j),...,e^{M_u}(x_{M_u},y_j)]^T$ for $j=1,...,N$, can now be converted to the appropriate set of hierarchical surpluses $\mathbf e^z$ through a standard linear transformation. Note that the system solves required to evaluate $\mathbf e_j$ involve the same coefficient matrix, but with multiple right hand sides, the computational effort of which is small. 

\vspace{1em}

The discretized augmented Lagrangian now takes the form 
\begin{align*}
L_c(q^{M_q,N},u^{M_u,N},\lambda^{M_u,N}) = &\phantom{+}\frac{1}{2}(\mathbf{u}^z)^T (S_\rho \otimes A_x^u) \mathbf u^z + \frac{\beta}{2}(\mathbf q^z)^T (S_\rho^\mathrm{mix} )\mathbf q^z\\
& + (\boldsymbol \lambda^z)^T S_{\rho,q}  \mathbf u^z + \frac{c}{2}(\mathbf e^z)^T (S_\rho \otimes A_x^u) \mathbf e^z, 
\end{align*}
while the gradients \eqref{eqn:aug_lag_grad_q} and \eqref{eqn:aug_lag_grad_u} of $L_c$ with respect to $q$ and $u$ are given by 
\begin{align}
D_q[L_c(q^{M_q,N},u^{M_u,N},\lambda^{M_u,N})]  = & \phantom{+}\beta (S_\rho^{\mathrm{mix}}\otimes A_x^q) \mathbf q^z + c S_{\rho,u}\mathbf e^z(\mathbf{q},\mathbf{u})\nonumber\\
& + S_{\rho,u}\boldsymbol \lambda^z - c S_{\rho,u} \mathbf e^z(\mathbf f)\label{eqn:aug_lag_aux_discrete_q} 
\end{align}
and 
\begin{align}
D_u[L_c(q^{M_q,N},u^{M_u,N},\lambda^{M_u,N})]  = & \phantom{+} (S_\rho\otimes A_x^u) \mathbf u^z + c S_{\rho,q}\mathbf e^z(\mathbf{q},\mathbf{u})\nonumber \\
& + S_{\rho,q}\boldsymbol \lambda^z - c S_{\rho,q} \mathbf e^z(\mathbf f) \label{eqn:aug_lag_aux_discrete_u}
\end{align}
respectively. The auxiliary problems \eqref{eqn:aug_lag_aux_q} and \eqref{eqn:aug_lag_aux_u} whose solutions yield updates for the parameter $q$ as well as the state $u$, can therefore be discretized in the form of two linear systems of size $M^q N$ and $M^u N$ respectively. These systems are where the bulk of the computational effort is spent. In our numerical computations, we employ the preconditioned conjugate gradient method.

\section{Numerical Results}\label{section: numerical examples}

In this section, we discuss three numerical examples to illustrate the use of the augmented Lagrangian method to estimate the statistical distribution of a spatially varying diffusion parameter $q$ from the measured output $\uhat$. In each case, we compute sample paths of $\uhat$ by solving \eqref{eqn: model} using sample paths of the exact parameter $q$ and a deterministic forcing term $f$, and perturbing the result slightly to account for measurement variability. We use a hierarchical basis of piecewise linear hat functions of the same order $L$ to interpolate $q,u,\lambda$ and $\uhat$. For the first two examples, the random variables that define the uncertain parameter are also used to express the model output and we construct the stochastic interpolant of $\uhat$ directly from that of $q$ by generating its sample paths at the appropriate sparse grid nodes. For the third example, we first compute a truncated KL expansion of $\uhat$, based on a randomly generated sample, and estimate the joint density of the pertinent random variables from which we then compute an interpolant. Throughout, we use the augmented Lagrangian with parallel splitting to effect the minimization. For the sake of regularization, we use a spatial discretization of $\uhat$ that is twice as fine as that of $q$ throughout. To assess the accuracy of our approximation, we compare the first few central moments of $	q$ with those of its approximation $\hat q$. In these examples, we did not enforce positivity of the constraint explicitly. 

\begin{example}\label{ex:1}
The first example serves to demonstrate the augmented Lagrangian method for a problem in 1 spatial- and 4 stochastic dimensions. The exact parameter $q$ and deterministic forcing term $f$ are defined over the domain $[0,1]$ by 
\begin{align*}
q(x,y) &= 2 + x^2 + \frac{1}{2}\sum_{i=1}^4 \cos(i\pi x) Y_i(\omega),\ \  Y_i(\omega) \sim \mathrm{i.i.d\ Uniform}([0,1]), \ \ \text{and}\\
f(x,y) &= 6x^2 - 2x + 4
\end{align*}
respectively. The manufactured solution $\uhat$ is perturbed by uniform random noise of relative size $\delta = 0.001$. We use $30$ elements for $q$ and $60$ for $u,\uhat$, and $\lambda$, a regularization term $\beta =$ 5e-5, an initial guess $q_0 = 1$, and terminate the program when the norm of the difference of successive iterates is within the tolerance 1e-5. Both sub-problems \eqref{eqn:aug_lag_aux_discrete_q} and \eqref{eqn:aug_lag_aux_discrete_u} are solved using a conjugate gradient routine with a relative residual tolerance of 1e-5. For this example, it is possible to plot and compare the sample paths of $q$ and $\hat q$ at the collocation points. Figure \ref{fig:ex1_sample_paths} shows that qualitatively, they indeed look similar. In Figure \ref{fig:ex1_moments}, we compare the first 4 central moments of $q$ and $\hat q$, which confirms that we are able to identify the statistical behavior of $q$ with a high accuracy (well within the magnitude of the noise added to the data). Table \ref{table:ex1_convergence} summarizes the convergence behavior of the algorithm.

\begin{figure}[ht]
	\centering
	\begin{subfigure}[t]{0.45\textwidth}
		\includegraphics[width = \textwidth]{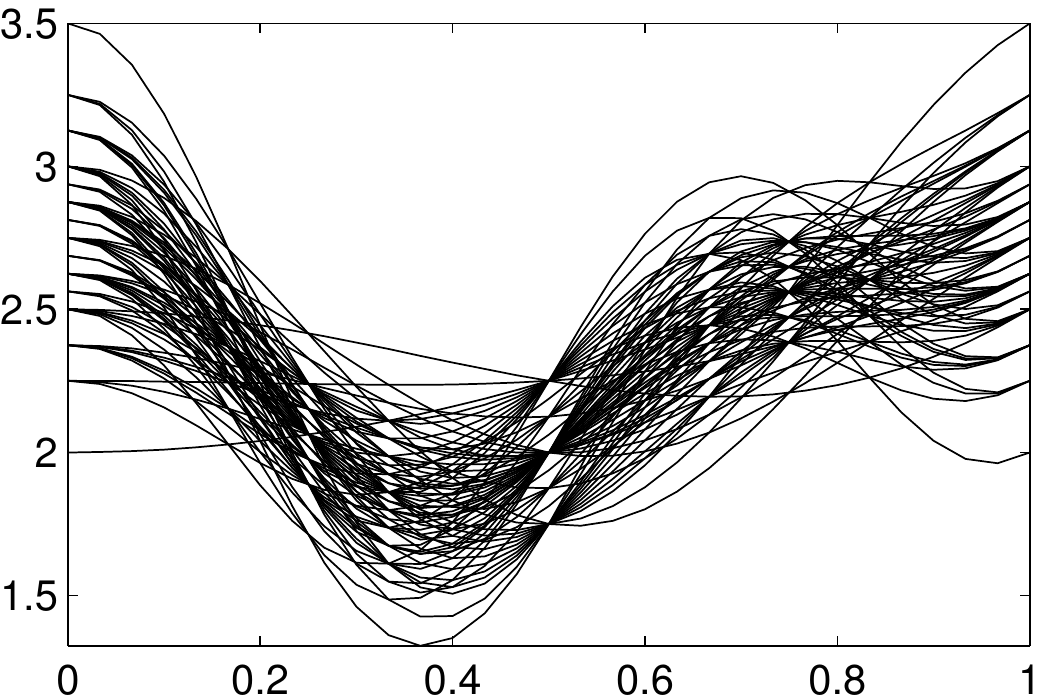}
		\caption{Exact parameter}
	\end{subfigure}%
	~
	\begin{subfigure}[t]{0.45\textwidth}
		\includegraphics[width=\textwidth]{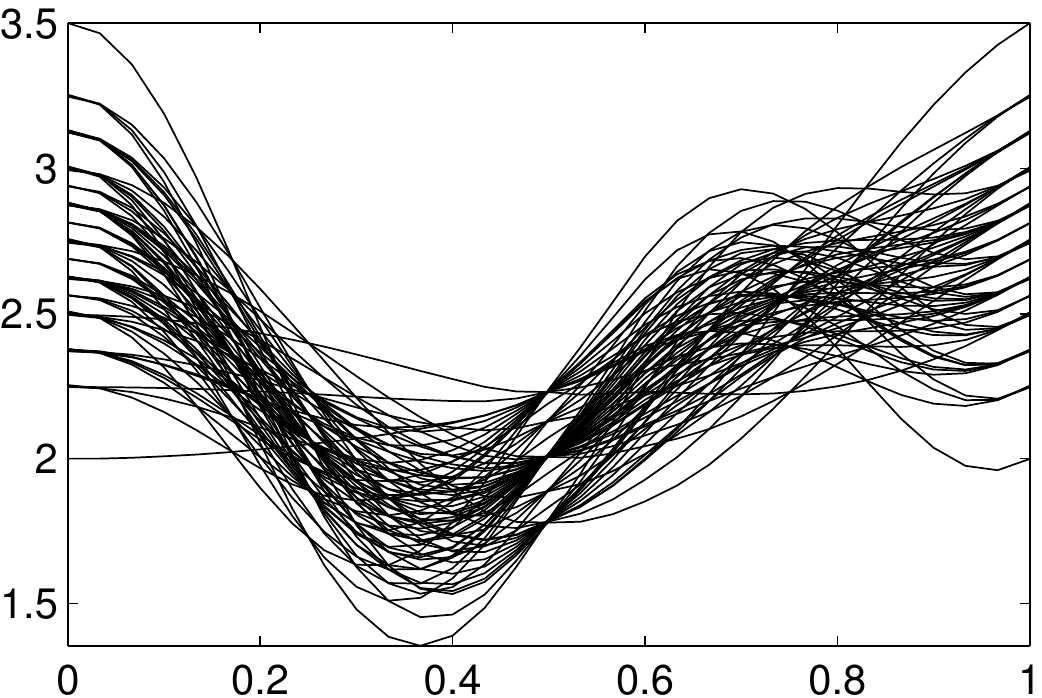}
		\caption{Identified parameter}
	\end{subfigure}
	\caption{Sample paths of the exact- and identified parameter.}
	\label{fig:ex1_sample_paths}
\end{figure}

\begin{figure}[ht]
\centering
\begin{subfigure}[t]{0.24\textwidth}
	\includegraphics[width=\textwidth]{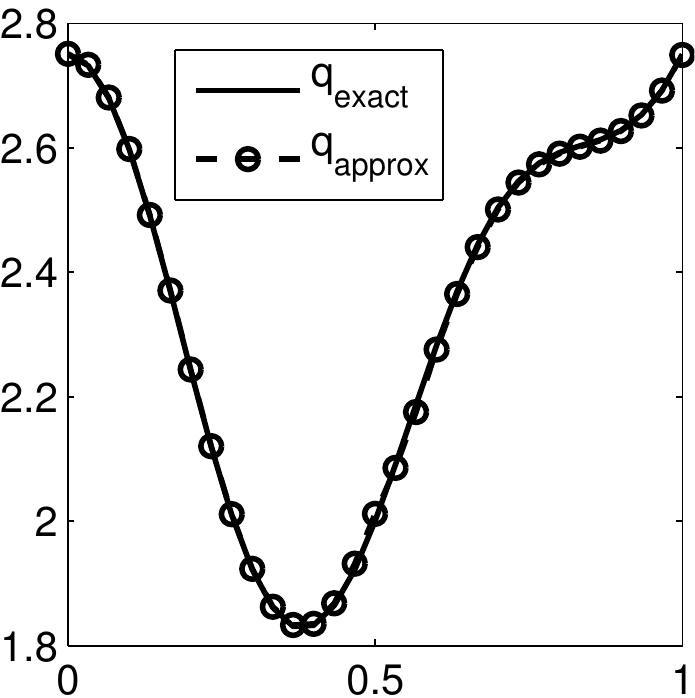}
	\caption{$\mu = \Exp[q(x)]$}
\end{subfigure}%
\begin{subfigure}[t]{0.24\textwidth}
	\includegraphics[width=\textwidth]{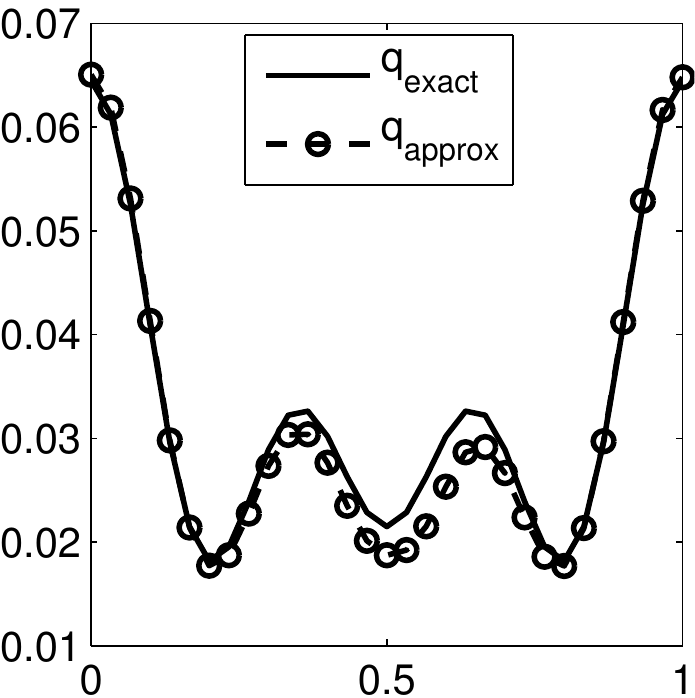}
	\caption{$\Exp\left[(q(x)-\mu(x))^2\right]$}
\end{subfigure}%
\begin{subfigure}[t]{0.24\textwidth}
	\includegraphics[width=\textwidth]{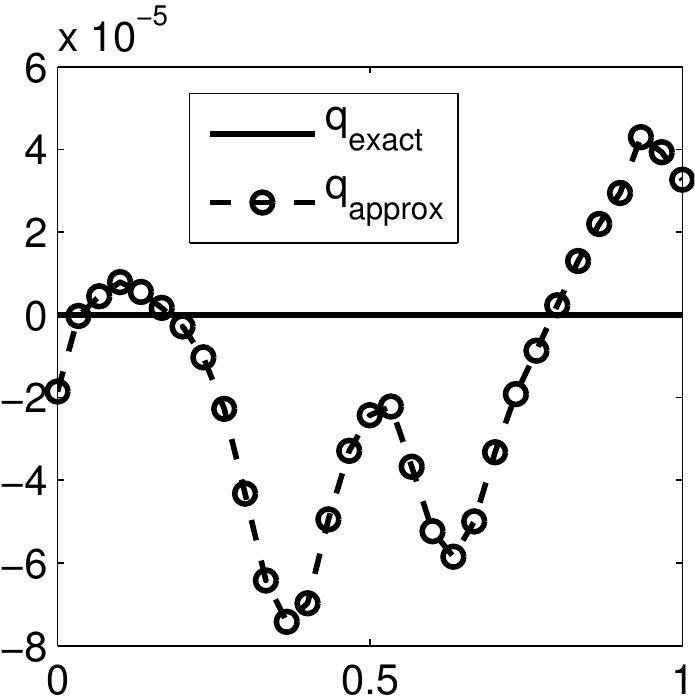}
	\caption{$\Exp\left[(q(x)-\mu(x))^3\right]$}
\end{subfigure}%
\begin{subfigure}[t]{0.24\textwidth}
	\includegraphics[width=\textwidth]{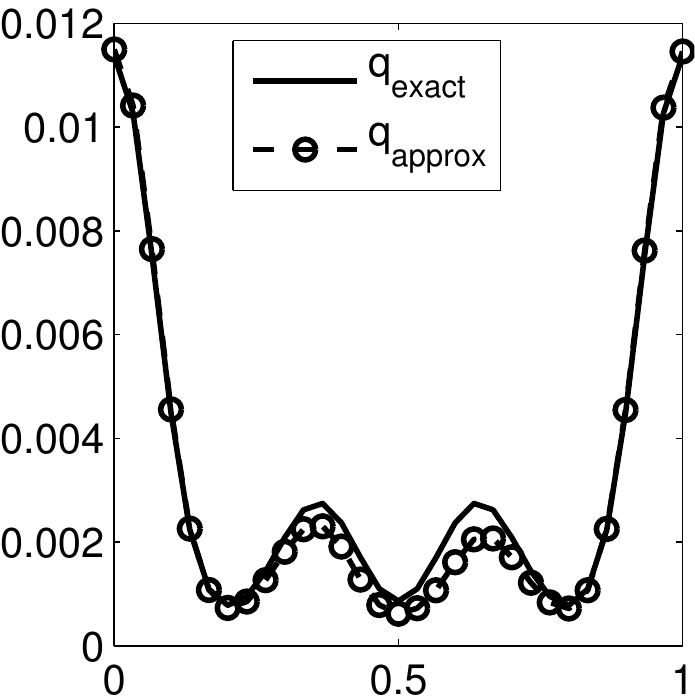}
	\caption{$\Exp\left[(q(x)-\mu(x))^4\right]$}
\end{subfigure}%
\caption{The first 4 central moments of $q$ and of its approximation $\hat{q}$.}
\label{fig:ex1_moments}
\end{figure}

\begin{table}[ht]
\centering
\begin{tabular}{|c|c|c|c|c|c|}
\hline
Step & \multicolumn{2}{|c|}{PCG Iterations} & $L^2$ error & Increments & Cost Functional\\
\hline
& \eqref{eqn:aug_lag_aux_q} & \eqref{eqn:aug_lag_aux_u} & $\|q-\hat q\|_{L^2}$ & $\|\hat{q}_{k}-\hat{q}_{k-1}\|_{L^2}$ & $J(q_k,u_k,\lambda_k)$\\
\hline 
1 & 1737 & 1246 & 1.9039 & -  &  1.7764e-20 \\
2 & 86 & 328 & 6.7864e-05 & 1.9019 & 5.4329e-05\\
3 & 25 & 118 &  9.2998e-05 & 2.7416e-06 & 5.3453e-05\\
\hline
\end{tabular}
\caption{Computational work and convergence diagnostics for for Example \ref{ex:1}.}
\label{table:ex1_convergence}
\end{table}
	
\end{example}

\begin{example} \label{ex:2}
As for deterministic inverse problems, the parameter $q$ may not be identifiable in certain spatial regions, due to the shape of the output for instance (see \cite{Kunisch1987}). This example investigates the role of regularization in this context. We chose a random output $\uhat$, most of whose sample paths have a zero gradient over a large area. Specifically, the deterministic forcing term $f$ is given by
\[
f(x_1,x_2) = - \nabla\cdot ( k(x_1,x_2)\nabla(w(x_1)w(x_2)) ),
\]
where 
\[
w(x) = \left\{
\begin{array}{ll}
9x^2 + 6x, & x \in [0,1/3]\\
1, & x\in (1/3,2/3) \\ 
-9x^2 + 12x - 3. & x\in [2/3,1]
\end{array}\right. ,
\]
and 
\[
k(x_1,x_2) = 2 + \sin(x_1^2x_2).
\]
The exact parameter $q$ is given by
\[
q(x_1,x_2,Y_1,Y_2,Y_3) = 2 + \sin(x_1^2x_2) + \frac{1}{8}\sum_{i=1}^3 \sin(i\pi x_1)\sin(i \pi x_2) ) Y_i, 
\]
where $Y_i \sim  \mathrm{i.i.d. Uniform}([-1,1])$, $i= 1,2,3$. We computed its approximation $\hat q$ on a uniform triangular mesh of 392 elements over the unit square, added the same level of noise $\delta$ as before, and interpolated in the stochastic component at level $L=4$. Figure \ref{fig:ex2_uhat} shows a typical sample path of $\uhat$. The problem was first solved using a regularization parameter $\beta =$1e-5, then again using $\beta = 1$e-3. In both cases the convergence tolerance was set to 1e-4 and the conjugate gradient tolerance was 1e-5. 

\begin{figure}[ht]
\centering
\begin{subfigure}[t]{0.23\textwidth}
\includegraphics[width = \textwidth]{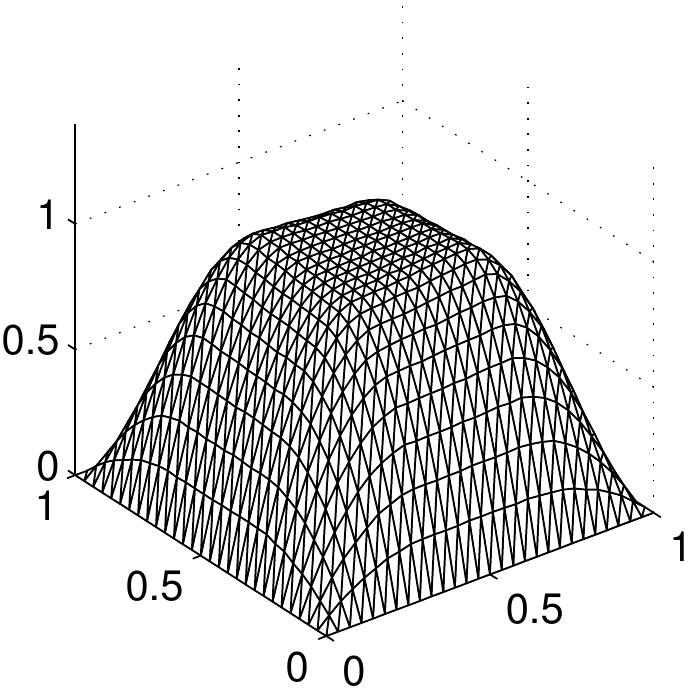}
\caption{A sample path of $\uhat$.}\label{fig:ex2_uhat}
\end{subfigure}%
\begin{subfigure}[t]{0.23\textwidth}
\includegraphics[width = \textwidth]{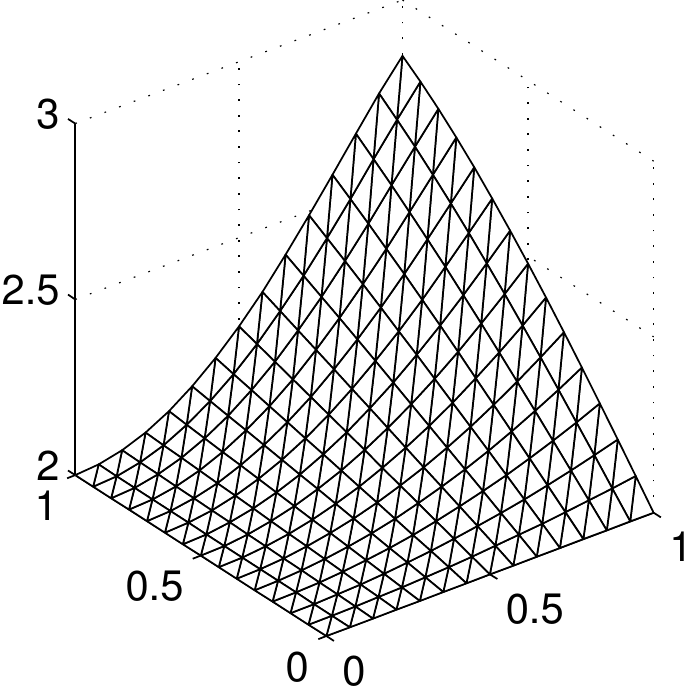}
\caption{$\Exp[q(x)]$}\label{fig:ex2_qmom_01_exact}
\end{subfigure}%
\begin{subfigure}[t]{0.23\textwidth}
	\includegraphics[width = \textwidth]{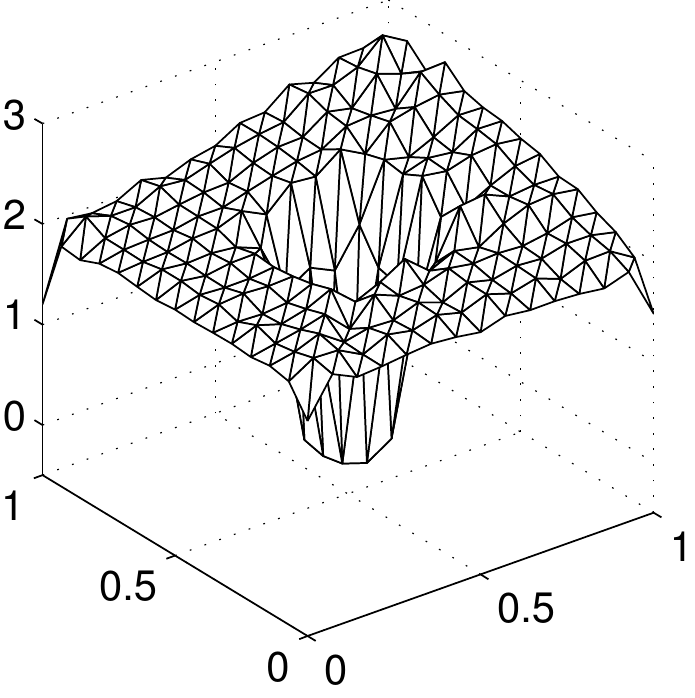}
\caption{$\Exp\left[\hat q(x)\right], \beta =$ 1e-5}
\label{fig:ex2_qmom_01_approx_noreg}
\end{subfigure}%
\begin{subfigure}[t]{0.23\textwidth}
	\includegraphics[width =\textwidth]{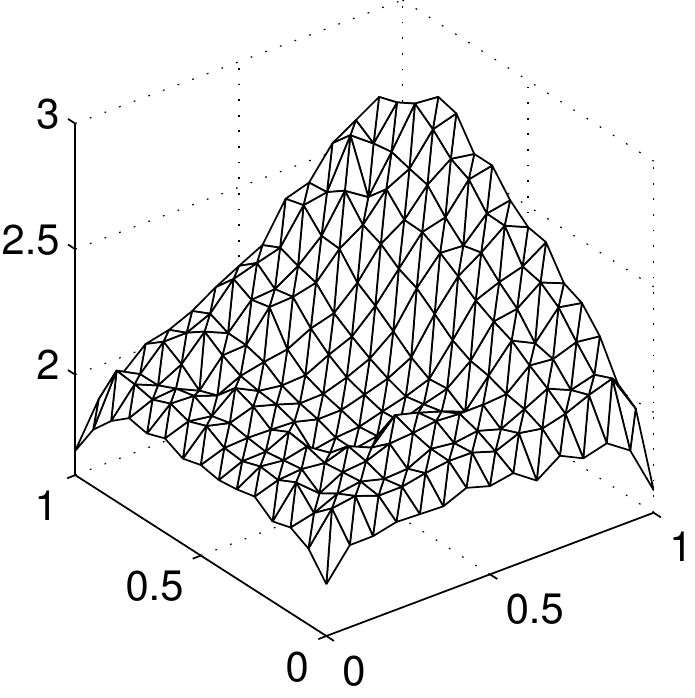}
\caption{$\Exp\left[\hat q(x)\right], \beta =$ 1e-3}
\label{fig:ex2_qmom_01_approx_reg}
\end{subfigure}
\label{fig:ex2_flat_uhat}
\end{figure}

Figures \ref{fig:ex2_qmom_01_exact}, \ref{fig:ex2_qmom_01_approx_noreg}, and \ref{fig:ex2_qmom_01_approx_reg} show the mean of $q$ and of $\hat q$ in each of these cases. Using a larger regularization parameter penalizes steep gradients, thereby improving the conditioning of the inverse problem, albeit at the cost of accuracy. Evidently, regularization continues to play a significant role in the estimation of uncertain parameters. Similar figures can be plotted for the higher order moments. Quantitative outputs of the algorithm are provided in Table \ref{table:ex2_convergence}.

\begin{table}[ht]
\centering
\begin{tabular}{|c|c|c|c|c|}
\hline
Step & $L^2$ error & Increments & Cost Functional & AL Functional\\
\hline
& $\|q-\hat q\|_{L^2}$ & $\|\hat{q}_{k}-\hat{q}_{k-1}\|_{L^2}$ & $J(q_k,u_k,\lambda_k)$ & $L(q_k,u_k,\lambda_k)$\\
\hline 
0 & 1.4083 & -   &  -2.1871e-17     & 0.0463     \\
1 & 0.0225 & 1.2559     & 0.0102 & 0.0130\\
2 & 0.0054 & 9.2e-3 & 0.0115  & 0.0118\\
3 & 0.0043 & 4.6058e-04  & 0.0117 &	0.0117		\\
4 & 0.0043 & 5.2789e-05 & 0.0116 &0.0116\\
\hline
\end{tabular}
\caption{Convergence table for Algorithm \ref{alg: augmented lagrangian method with splitting} applied to Example \ref{ex:2} with $\beta=$1e-3.}
\label{table:ex2_convergence}
\end{table}
\end{example}

\begin{example}\label{ex3}

For this example, the random variables used to express the identified parameter are estimated from sample paths of the model output $\uhat$. The deterministic forcing term satisfies
\[
f(x_1,x_2) = - \nabla\cdot ( (4 + x_1x_2)\nabla \sin(\pi x_1)\sin(\pi x_2) ),
\]
while
\begin{align*}
q(x_1,x_2,y_1,y_2,y_3) = &\ 4 + x_1 x_2 + 0.5\sin(\pi x_1)\sin(\pi x_2)Y_1 \\
& + 0.25\cos(0.5\pi x_1)\sin(0.5\pi x_2)Y_2 + 0.25\cos(\pi x_1)\cos(\pi x_2)Y_3,
\end{align*}
where $Y_i \sim \mathrm{i.i.d. Uniform([-1,1])}$. Using random samples of these input parameters, we generated 1000 sample paths of $\uhat$, which we then decomposed according to the method outlined in Section \ref{section: discretization}. No additional noise was added to the sample paths. For this problem, 2 KL expansion terms suffice to represent the sample $\uhat$ so that the remaining expansion terms contribute less than tol=1e-7 to the field's variance. We express each random variable $Y_i,i=1,2$ as the inverse image of a uniform random variable under its empirical cumulative distribution function (cdf). The appropriate graphs are shown in Figure \ref{fig:ex3_kl_trunc}. 

\begin{figure}[ht]
\centering
\begin{subfigure}[t]{0.23\textwidth}
\includegraphics[width=\textwidth]{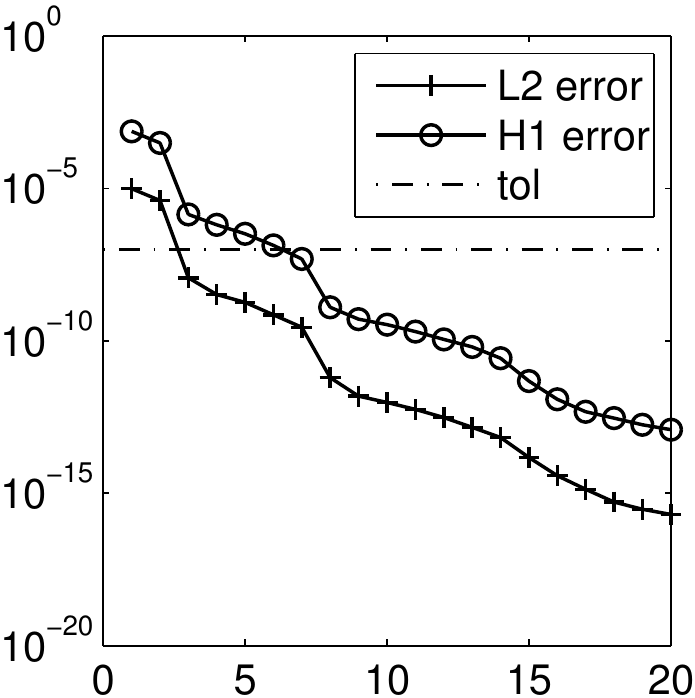}
\caption{KL truncation}
\end{subfigure}%
\begin{subfigure}[t]{0.23\textwidth}
\includegraphics[width=\textwidth]{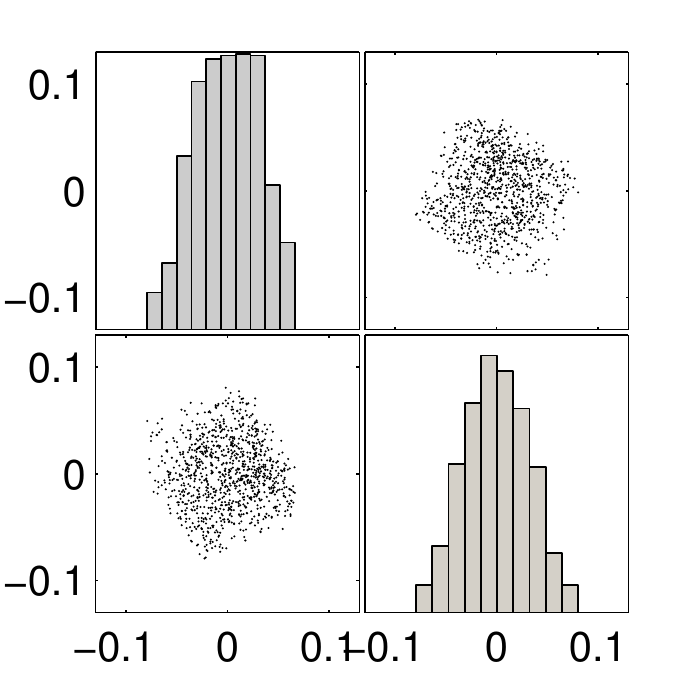}
\caption{Scatterplot$(Y_1,Y_2)$}
\end{subfigure}%
\begin{subfigure}[t]{0.23\textwidth}
\includegraphics[width=\textwidth]{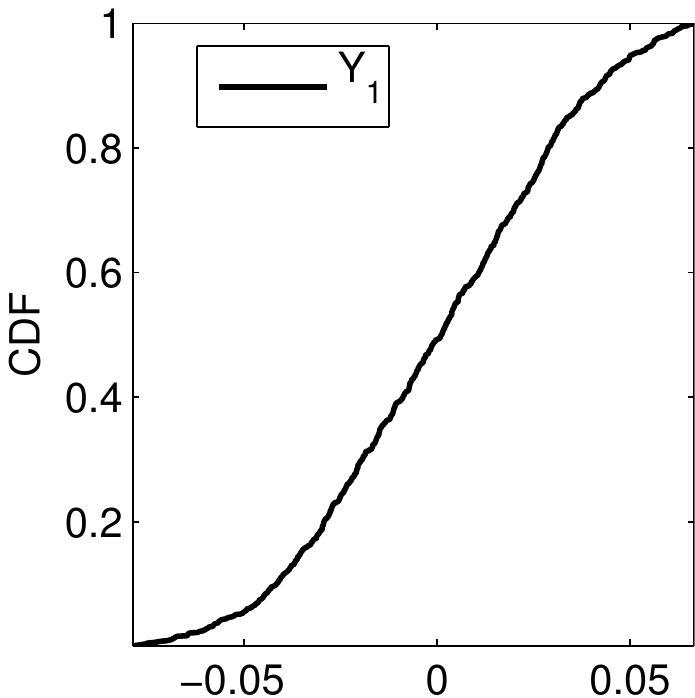}
\caption{CDF $Y_1$}
\end{subfigure}%
\begin{subfigure}[t]{0.23\textwidth}
\includegraphics[width=\textwidth]{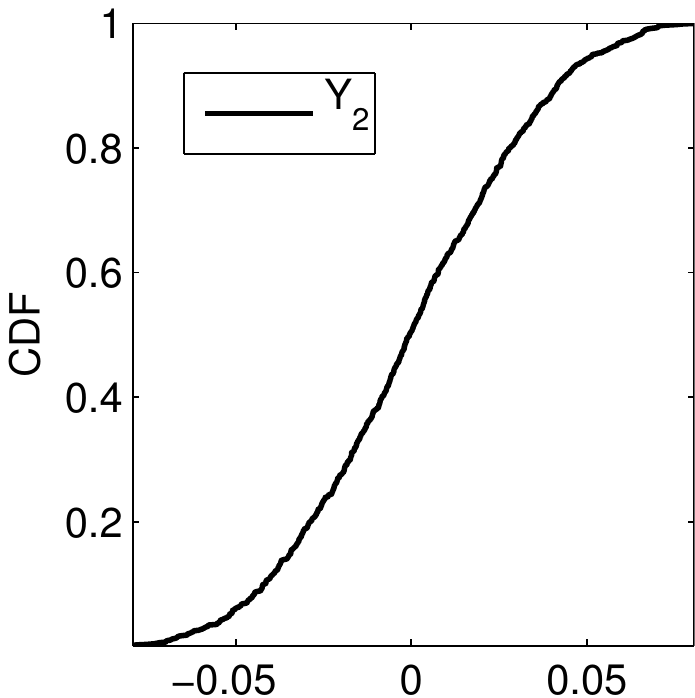}
\caption{CDF $Y_2$}
\end{subfigure}
\caption{Sparse grid interpolation of $\uhat$ based on a random sample of 1000 paths.}
\label{fig:ex3_kl_trunc}
\end{figure}

As in Example \ref{ex:2}, we discretize $q$ using a uniform spatial mesh of 392 elements. In addition, we choose a sparse grid interpolation level $L=4$. We use a regularization term $\beta=$1e-5, and terminate the optimization algorithm when the $L^2$ norm of successive iterates is within the tolerance level of 1e-5. For the conjugate gradient subroutines, we use a tolerance of 1e-6. As before, we compare the central moments of the identified parameter $\hat q$ with those of its exact counterpart $q$ to assess  its accuracy. Figure \ref{fig:ex3_moments} shows that, qualitatively, the estimate is good. Since the random variables used to express $\hat q$ differ from $Y_1$ and $Y_2$, it is impossible to compute the exact error as part of the optimization run. We nevertheless record relevant convergence diagnostics in Table \ref{table:ex3_convergence}. 

\begin{figure}[ht]
\centering
\begin{subfigure}[t]{0.23\textwidth}
\includegraphics[width=\textwidth]{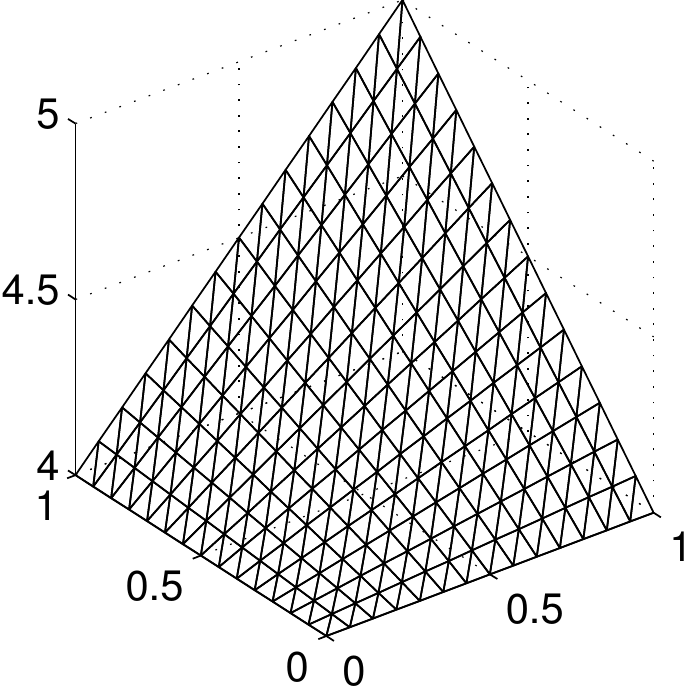}
\end{subfigure}%
\begin{subfigure}[t]{0.23\textwidth}
\includegraphics[width=\textwidth]{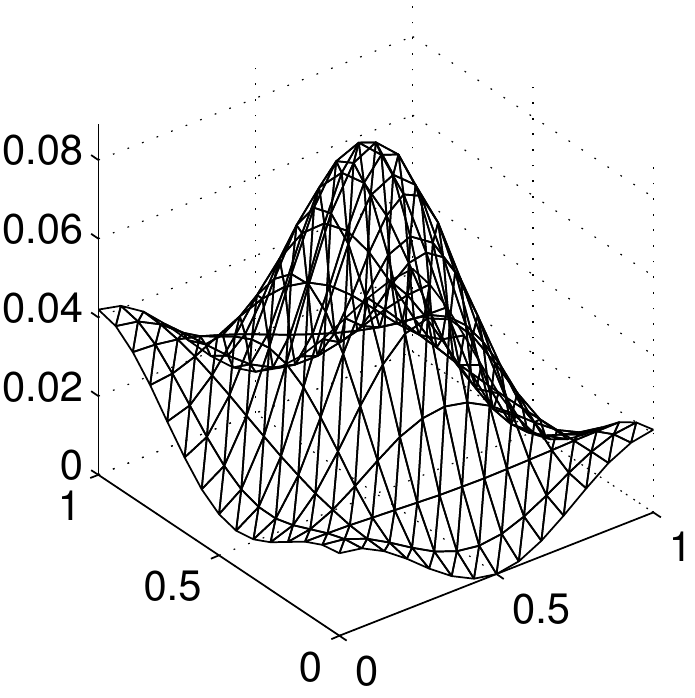}
\end{subfigure}%
\begin{subfigure}[t]{0.23\textwidth}
\includegraphics[width=\textwidth]{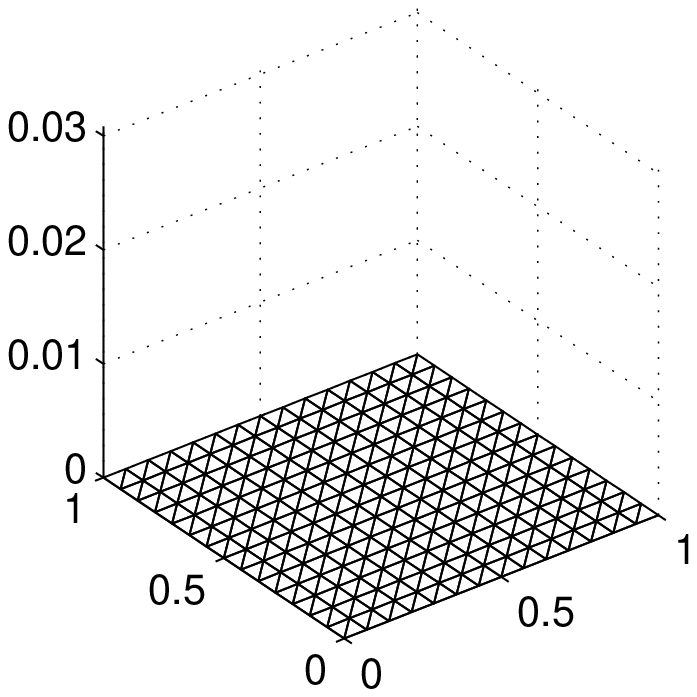}
\end{subfigure}%
\begin{subfigure}[t]{0.23\textwidth}
\includegraphics[width=\textwidth]{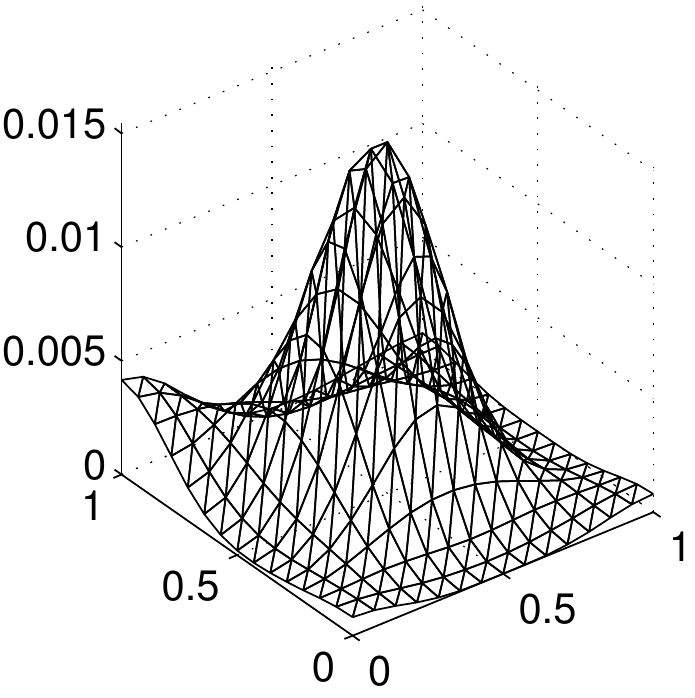}
\end{subfigure}\\
\begin{subfigure}[t]{0.23\textwidth}
\includegraphics[width=\textwidth]{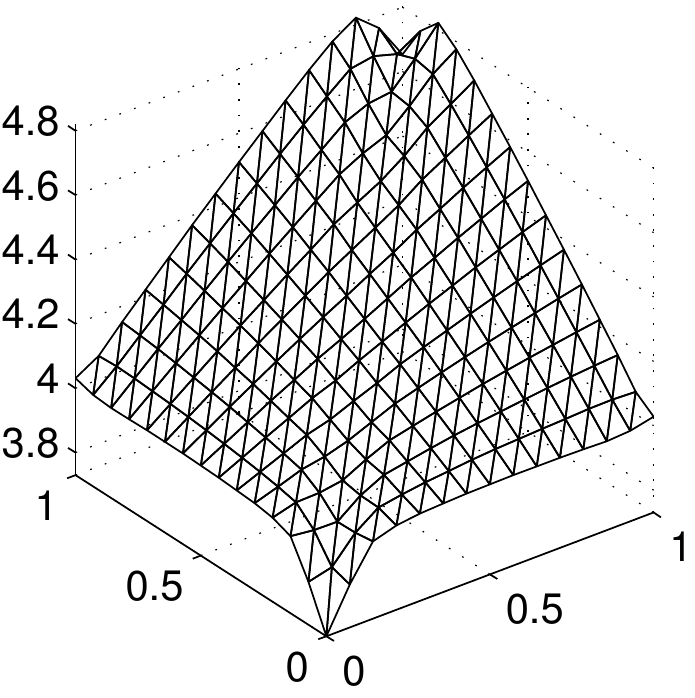}
\caption{$\mu = \Exp[q(x)]$}
\end{subfigure}%
\begin{subfigure}[t]{0.23\textwidth}
\includegraphics[width=\textwidth]{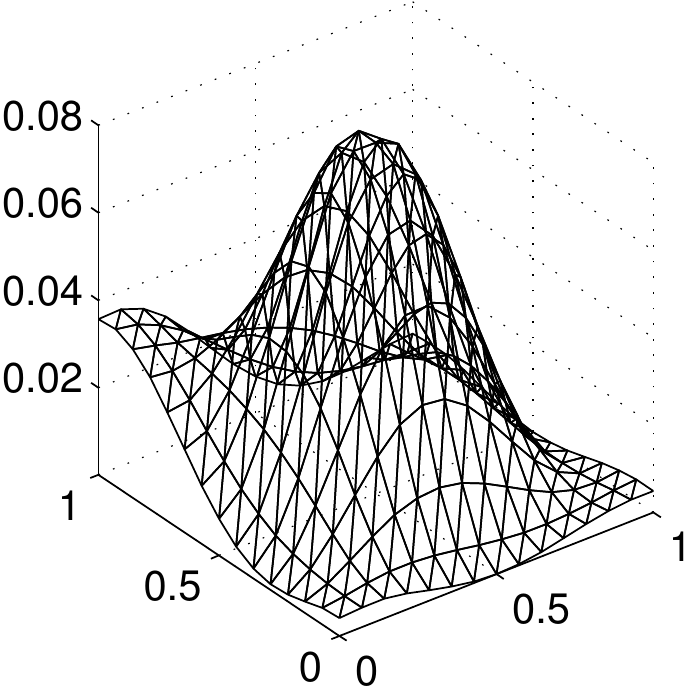}
\caption{$\Exp\left[(q(x)-\mu(x))^2\right]$}
\end{subfigure}%
\begin{subfigure}[t]{0.23\textwidth}
\includegraphics[width=\textwidth]{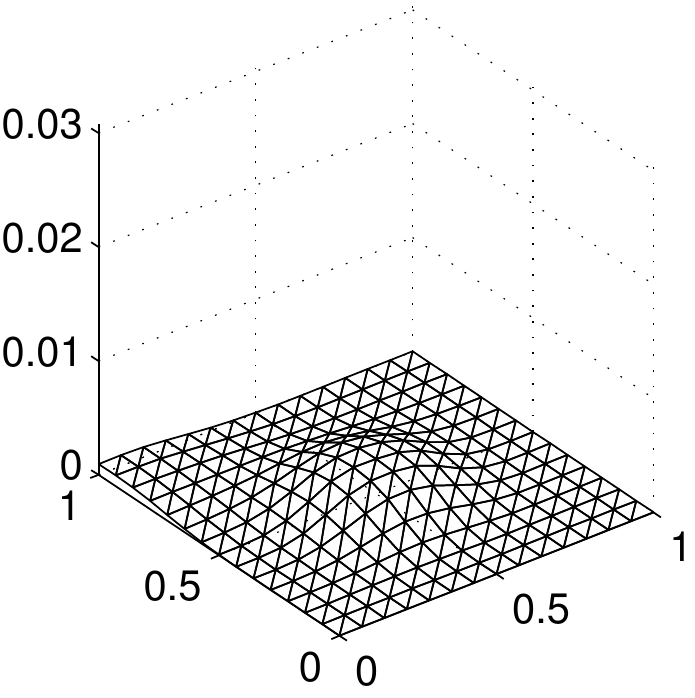}
\caption{$\Exp\left[(q(x)-\mu(x))^3\right]$}
\end{subfigure}%
\begin{subfigure}[t]{0.23\textwidth}
\includegraphics[width=\textwidth]{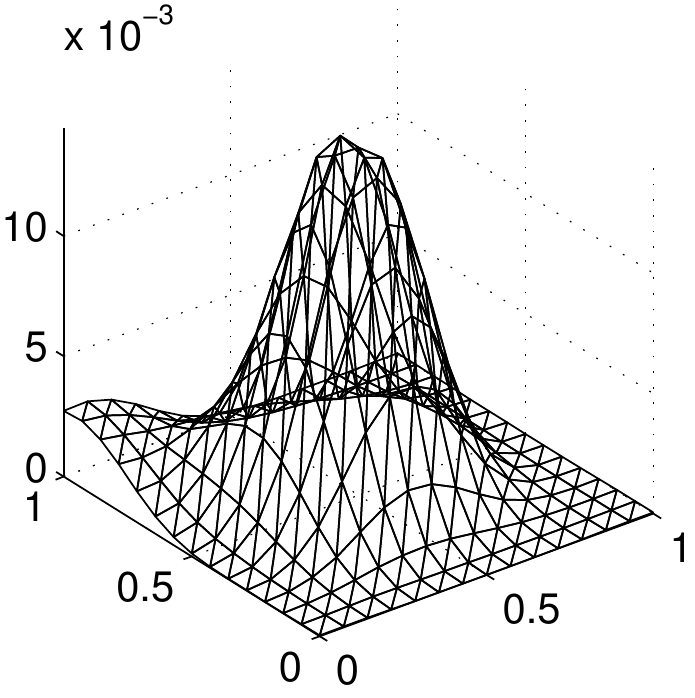}
\caption{$\Exp\left[(q(x)-\mu(x))^4\right]$}
\end{subfigure}
\caption{The first 4 central moments of the exact parameter $q$ (top row) and those of the identified parameter $\hat q$ (bottom row).}
\label{fig:ex3_moments}
\end{figure}

\begin{table}[ht]
\centering
\begin{tabular}{|c|c|c|c|}
\hline
Step & Increments & Cost Functional & AL Functional\\
\hline
& $\|\hat{q}_{k}-\hat{q}_{k-1}\|_{L^2}$ & $J(q_k,u_k,\lambda_k)$ & $L(q_k,u_k,\lambda_k)$\\
\hline 
0 & - & -   &  -        \\
1 &    10.6826 & 6.1409e-05 & 6.1461e-05\\
2 & 3.7636e-05 & 5.3654e-05 & 5.3680e-05\\
3 & 1.2276e-05 & 5.0058e-05 & 5.0075e-05\\
4 & 4.3860e-06 & 4.8018e-05 & 4.8029e-05\\
\hline
\end{tabular}
\caption{Convergence table for Algorithm \ref{alg: augmented lagrangian method with splitting} applied to Example \ref{ex3}.}
\label{table:ex3_convergence}
\end{table}

\end{example}

\section{Conclusion}

In this paper we have formulated a fairly general variational framework for the estimation of spatially distributed, uncertain diffusion coefficients in stationary elliptic problems, based on statistical measurements of the model output. In contrast to the Bayesian approach, we used a parametrization of the coefficient in terms of a finite number of variables, allowing us to not only estimate the statistical mismatch between the predicted- and observed output, but also to determine the perturbations of $q$ that will result in a decrease in the degree of mismatch. In light of the potential size in the number of degrees of freedom, the computation of quantities such as steepest descent directions, or cost functional evaluations may require considerable computational cost. We are currently investigating ways to reduce the computational overhead, through parallelization \cite{vanwyk2014}, multigrid methods, or the use of sensitivity information \cite{Borggaard2013}. 
\bibliographystyle{siam}
\bibliography{bibliography_uncertainparid}

\end{document}